\documentclass[11pt,a4paper]{amsart}
\usepackage[margin=3cm]{geometry}
\usepackage{amsmath,amsthm,amssymb}
\usepackage{tikz}
\usepackage{caption}

\usetikzlibrary{arrows,decorations.pathmorphing,decorations.pathreplacing}

\tikzset{vertex/.style={circle,fill=black,inner sep=1pt,outer sep=2pt},
         mvertex/.style={rectangle,draw=black,thick,inner sep=2pt,outer sep=2pt},
         tvertex/.style={inner sep=1pt,font=\scriptsize},
         unvertex/.style={circle,fill=white,draw=white,inner sep=1pt}, 
         WhiteBK/.style={circle,fill=white,inner sep=0pt},
         fill1/.style={fill=black!20,draw=black!20},
         fill1/.style={fill=black!20,draw=black!20},
         fill2/.style={fill=black!40,draw=black!40},
         fill12/.style={fill=black!60,draw=black!60},
         fillwhite/.style={fill=white,draw=white},
         >=stealth',
         leadsto/.style={-angle 90,decorate,decoration=snake,very thick},
         cut/.style={decorate,decoration=saw,very thick}
}

\newcommand{\replacevertex}[3][fill=white,draw=white]
 {
  \node at #2 [#1,circle,inner sep=5pt] {};
  \node #2 at #2 #3;
 }

\newtheorem{theorem}{Theorem}[section]

\newtheorem{corollary}[theorem]{Corollary}

\newtheorem{proposition}[theorem]{Proposition}

\theoremstyle{definition}
\newtheorem{definition}[theorem]{Definition}
\newtheorem{construction}[theorem]{Construction}
\newtheorem{setup}[theorem]{Setup}
\theoremstyle{definition}
\newtheorem{example}[theorem]{Example}
\theoremstyle{remark}
\newtheorem*{remark}{Remark}

\theoremstyle{definition}

\DeclareMathOperator{\add}{add}
\DeclareMathOperator{\Aut}{Aut}
\newcommand{\comp}{\operatorname{\scriptstyle\circ}}

\DeclareMathOperator{\Ext}{Ext}
\DeclareMathOperator{\End}{End}

\DeclareMathOperator{\Hom}{Hom}

\newcommand{\op}{{\operatorname{op}\nolimits}}
\DeclareMathOperator{\Pot}{Pot}

\newcommand{\Sum}{\sum\limits} 

\newcommand{\extto}{\xrightarrow}

\newcommand{\from}{\colon \!}
\newcommand{\lperp}[1]{\sideset{^\perp}{}{\operatorname{#1}}}
\newcommand{\rperp}[1]{\sideset{}{^\perp}{\operatorname{#1}}}
\newcommand{\shift}[1]{\langle #1 \rangle}

\newcommand{\dT}{ \sideset{_\ocT}{}{ \operatorname{T} } }
\newcommand{\downstairs}[1]{ \sideset{_\ocT}{}{ \operatorname{#1} } }
\newcommand{\uT}{ \sideset{_\cT}{}{ \operatorname{T} } }
\newcommand{\upstairs}[1]{ \sideset{_\cT}{}{ \operatorname{#1} } }

\newcommand{\bbA}{ \mathbb{A}}  
\newcommand{\bbD}{ \mathbb{D}}  
\newcommand{\bbE}{ \mathbb{E}}  
  
\newcommand{\bbZ}{ \mathbb{Z}}  

\newcommand{\cA}{ \ensuremath{ {\mathcal A} } }

\newcommand{\cC}{ \ensuremath{ {\mathcal C} } }
\newcommand{\cD}{ \ensuremath{ {\mathcal D} } }
\newcommand{\cE}{ \ensuremath{ {\mathcal E} } }

\newcommand{\cH}{ \ensuremath{ {\mathcal H} } }

\newcommand{\cJ}{ \ensuremath{ {\mathcal J} } }

\newcommand{\cP}{ \ensuremath{ {\mathcal P} } }

\newcommand{\cT}{ \ensuremath{ {\mathcal T} } }
\newcommand{\cU}{ \ensuremath{ {\mathcal U} } }

\newcommand{\cX}{ \ensuremath{ {\mathcal X} } }

\newcommand{\cZ}{ \ensuremath{ {\mathcal Z} } }

\newcommand{\wpi}{ \widetilde{\pi}}
 
\newcommand{\oa}{ \overline{a}}
\newcommand{\ob}{ \overline{b}}

\newcommand{\od}{ \overline{d}}

\newcommand{\om}{ \overline{m}}

\newcommand{\ocD}{ \overline{\cD}}

\newcommand{\ocT}{ {\overline{\cT}} }
\newcommand{\ocU}{ {\overline{\cU}} }
\newcommand{\ocZ}{ \overline{\cZ}}

\newcommand{\oD}{ \overline{D}}
\newcommand{\oT}{ \overline{T}}
\newcommand{\oQ}{ \overline{Q}}
\newcommand{\oX}{ \overline{X}}
\newcommand{\oY}{ \overline{Y}}
\newcommand{\oZ}{ \overline{Z}}

\newcommand{\laT}{{\overleftarrow{\mathrm T}}}
\newcommand{\laY}{{\overleftarrow{Y}}}

\makeatletter
\@namedef{subjclassname@2010}{%
  \textup{2010} Mathematics Subject Classification}
\makeatother

\numberwithin{figure}{section}
\numberwithin{table}{section}
\numberwithin{equation}{section}

\title[Mutating loops and 2-cycles]{Mutating loops and
  2-cycles in 2-CY triangulated categories}
\date{\today}

\author[Bertani-{\O}kland]{Marco Angel Bertani-{\O}kland}
\author[Oppermann]{Steffen Oppermann}
\address{Institutt for matematiske fag\\ NTNU\\ 7491 Trondheim\\ Norway}
\email{Marco.Tepetla@math.ntnu.no}
\email{Steffen.Oppermann@math.ntnu.no}


\begin{document}

\begin{abstract}
We derive an algorithm for mutating quivers of 2-CY tilted algebras that have loops and 2-cycles, under certain specific conditions. Further, we give the classification of the 2-CY tilted algebras coming from standard algebraic 2-CY triangulated categories with a finite number of indecomposables. These form a class of algebras that satisfy the setup for our mutation algorithm. 
\end{abstract}

\maketitle
\section{Introduction and Results}

 Mutation has played an important role in representation theory during the recent years, especially in tilting and cluster-tilting theory. For instance, let $\cH$ be a hereditary abelian $k$-category over a field $k$, with finite dimensional $\Hom$-spaces and $\Ext^{1}$-spaces having a tilting object $T$. If $\cH$ has no nonzero projective (or nonzero injective) objects, we know that every almost complete tilting object (that is, a tilting object where one indecomposable summand is removed) has exactly two complements (see \cite{Hubner2,BMRRT,BOW2}). Thus we can always replace any indecomposable summand of $T$, to obtain a new tilting object $T'$. This procedure is called (tilting) mutation in $\cH$. Unfortunately, this procedure does not work in general if $\mathcal{H}$ has nonzero projectives.

Fortunately, there exists a generalization of tilting theory, where mutation is always possible: {\em cluster-tilting theory}. The approach goes as follows: Trying to model cluster algebras from a categorical point of view, the authors in \cite{BMRRT} introduced the cluster category $\cC_H$ for a hereditary algebra $H$ (see also \cite{CCS1} for the $\bbA_n$ case), and more generally for a hereditary category $\cH$ with a tilting object. The cluster category comes equipped with a class of objects called {\em cluster-tilting objects}. There is a mutation of cluster-tilting objects in cluster categories, and more generally, in $\Hom$-finite triangulated $2$-Calabi-Yau categories ($2$-CY for short) over an algebraically closed field $k$ (see \cite{IY}).

Associated to a tilting object $T$ in $\cH$ is the endomorphism algebra $\End_{\cH}(T)$, called {\em quasi-tilted algebra} (\cite{HRS}). The mutation of tilting objects induces  a mutation of quasi-tilted algebras (see \cite{Hubner}). Similarly, associated to a cluster-tilting object $T$ in $\cC$ is the endomorphism algebra $\End_{\cC}(T)$, called {\em cluster-tilted algebra} (\cite{BMR2}). The mutation of cluster-tilting objects also induces a mutation of cluster-tilted algebras, and further a mutation of their quivers (see \cite{BIRSm}). It coincides (see \cite{BMR3,BIRSc}) with the quiver mutation rule given by S. Fomin and A. Zelevinsky in \cite{FZ1} in their theory of cluster algebras.

 The main limitation of this quiver mutation rule is, that one only mutates at vertices not lying in loops or 2-cycles. At the categorical level, there is no such restriction. Therefore one would expect that there is a way to generalize the quiver mutation rule to vertices lying on loops and 2-cycles. The aim of this paper is to give a procedure for mutating quivers of $2$-CY tilted algebras (the endomorphism rings of cluster-tilting objects in $2$-CY triangulated categories) at vertices with loops and 2-cycles, under some special conditions.
 
The setup is the following: We assume to have a Galois covering $\pi\from \cT \to \ocT$ of algebraic 2-CY triangulated categories, where no cluster-tilting object in $\cT$ has  loops or 2-cycles. For a (basic) cluster-tilting object $\dT$ in $\ocT$ having loops and/or 2-cycles, we denote by $\uT$ its lift in $\cT$. Then $\uT$ is a cluster-tilting object in $\cT$ (see Proposition~\ref{prop.lift.cto}). 

 First, we develop a procedure to replace the fibre of an indecomposable summand of $\dT$ in $\uT$, in order to  obtain a new cluster-tilting object $\uT'$ in $\cT$. Then we show that $\pi(\uT')$ coincides with replacing the corresponding indecomposable summand of $\dT$ in $\ocT$.

 Second, we develop a corresponding procedure at the level of quivers. That is, if we denote by $\oQ$ (resp.\ $Q$) the quiver of the $2$-CY tilted algebra associated to $\dT$ (resp.\ $\uT$), we give a method to mutate at any given vertex $v$ of $\oQ$ by mutating its cover $Q$ at the fibre $\pi^{-1}(v)$, where $\pi\from Q \to \oQ$ also denotes the induced covering morphism of quivers. 

 Finally, we give the classification of the $2$-CY tilted algebras of finite type. We show that this class of algebras satisfies the setup for our mutation algorithm, and organize them according to their mutation classes.

The constructions above rely heavily on a reduction technique by \cite{IY}, and a generalized mutation rule for algebraic $2$-CY triangulated categories by \cite{Palu}.

The paper is organized as follows.

In Section~\ref{section.background} we define the notation we use and recall some basic results on mutation of quivers, quivers with potentials, cluster categories, coverings, Palu's generalized mutation rule for algebraic $2$-CY triangulated categories, and Iyama-Yoshino's reduction technique.

In Section~\ref{section.mutation.cycles}, we develop the theory to replace, at the same time, several summands of a cluster-tilting object in an algebraic $2$-CY triangulated category. Then we derive our rule to mutate at (minimal) oriented cycles of $2$-CY tilted algebras (due to the lengh of the calculations, they are postponed to Appendix~\ref{appendix.ugly}). Furthermore, we prove that mutating at cycles is equivalent to a sequence of FZ-mutations.

In Section~\ref{section.mut.loops.2cycles} we present the algorithm to mutate quivers of $2$-CY tilted algebras having loops and/or 2-cycles.
  
Finally, in Section~\ref{section.2CY.finite.type} we present the classification of the $2$-CY tilted algebras coming from standard algebraic 2-CY triangulated categories with a finite number of indecomposables. By using our mutation procedure, we are able to mutate at any vertex in the quivers of these algebras.

\section{Background} \label{section.background}
\subsection{Conventions} \label{sec.conventions}
Fix $k$ an algebraically closed field. When we say that $\cC$ is a category, we always assume that $\cC$ is $k$-linear additive with finite dimensional morphism spaces.

 Let $\cC$ be a category.  We denote by $\cC(X,Y)$ or by $(X,Y)$ the set of morphisms from $X$ to $Y$ in $\cC$. An {\em ideal} $I$ of $\cC$ is an additive subgroup $I(X,Y)$ of $\cC(X,Y)$ such that $fgh\in I(W,Z)$ whenever $f\in\cC(W,X)$, $g\in I(X,Y)$, and $h\in\cC(Y,Z)$. For an ideal $I$ of $\cC$, we write $\cC/I$ for the category whose objects are the objects of $\cC$ and whose morphisms are given by $\cC/I(X,Y)=\cC(X,Y)/I(X,Y)$ for $X,Y \in \cC/I$. 

When we say that $\cD$ is a subcategory of $\cC$, we always mean that $\cD$ is a full subcategory that is closed under isomorphisms, direct sums, and direct summands. We denote by $[\cD]$ the ideal of $\cC$ consisting of morphisms that factor through objects in $\cD$. Thus we can form the category $\cC/[\cD]$. 

A morphism $f$ is said to be {\em right minimal} if it does not have a summand of the form $X_0 \to 0$ as a complex, for a nonzero object $X_0\in\cC$. For a subcategory $\cD$ of $\cC$, a morphism $f$ is called a {\em right $\cD$-approximation} of $Y\in\cC$ if $X\in \cD$ and
\[
  \cC(-,X) \extto{(-,f)} \cC(-,Y) \to 0, 
\]
is an exact sequence of functors on $\cD$.  We say that a right $\cD$-approximation is {\em minimal} if it is right minimal. A subcategory $\cD$ is called a {\em contravariantly finite subcategory} of $\cC$ if any $Y\in\cC$ has a right $\cD$-approximation. Dually one defines a {\em left $\cD$-approximation} and a {\em covariantly finite subcategory}. A contravariantly and covariantly finite subcategory is said to be {\em functorially finite}.

Let $\cX$ be a subcategory of a category $\cT$. Define $\rperp{\cX}$ to be the subcategory of all $T\in\cT$ such that $(\cX,T)=0$. Dually, $\lperp{\cX}=\{T\in\cT | (T,\cX)=0\}$. 

A $k$-linear autofunctor $\nu\from \cT \to \cT$ of a triangulated category $\cT$ is called a {\em Serre functor} of $\cT$ if there is a functorial isomorphism $(X,Y)\simeq D(Y,\nu X)$ for any $X,Y\in\cT$, where $D$ denotes the usual $k$-duality. If $\cT$ has a Serre functor, then it is unique (up to natural isomorphism). We say that $\cT$ is {\em $n$-Calabi-Yau} ($n$-CY for short) for an integer $n\in\bbZ$ if $\nu=[n]$.

A triangulated category is {\em algebraic} if it is triangle equivalent to the stable category of a Frobenius exact category. All triangulated categories occuring throughout this paper are assumed to be algebraic.

\subsection{Quiver mutation}\label{section.quiver_mutation}

 Let $Q$ be a finite quiver with vertices $1,\ldots, n$ having no loops or $2$-cycles. Let $q_{i,j}$ denote the number of arrows from $i$ to $j$ minus the number of arrows from $j$ to $i$ in $Q$. The terms $q_{i,j}$ can be collected in a matrix called the {\em skew-symmetric matrix} associated to $Q$. In this setting, Fomin-Zelevinsky \cite{FZ1} defined a mutation rule on quivers. For a vertex $\ell$, we obtain the new quiver $\mu_{\ell}(Q)$ as follows: The skew-symmetric matrix $(q_{i,j}')$ associated to $\mu_{\ell}(Q)$ is given by
\[
 q'_{i,j} =
 \begin{cases}
     -q_{i,j} & \text{if $i=\ell$ or $j=\ell,$} \\
     q_{i,j} + \frac{|q_{i,\ell}|q_{\ell,j} + q_{i,\ell}|q_{\ell,j}|}{2} & \text{otherwise.}\\
 \end{cases}
\]
We say that $Q$ and $\mu_{\ell}(Q)$ are mutations of one another. Observe that $\mu_{\ell} (\mu_{\ell}(Q))\simeq Q$. The collection of all quivers which are iterated mutations of $Q$ is called the {\em mutation class} of $Q$.

\subsection{Quivers with potentials}\label{sec.QP} We follow \cite{DWZ}. Let $Q$ be a quiver (possibly with loops and $2$-cycles). Denote by $Q_{0}$ its set of vertices (or, equivalently, the paths of length zero) and by $Q_{i}$ the paths of length $i$ in $Q$, where $i$ is a positive integer. Let $kQ_{i}$ be the  $k$-vector space with basis $Q_{i}$ and denote by $kQ_{i}^{c}$ the subspace of $kQ_{i}$ generated by all the cycles. The complete path algebra of $Q$ is then defined as
\[ \widehat{kQ} = \prod_{i\ge 0} kQ_{i}, \]
that is, the completion of the path algebra $kQ$ with respect to the ideal generated by all the arrows. For an integer $m\ge 2$ let  
\[ \Pot_{m}(kQ) =\prod_{i\ge m}kQ_{i}^c .\]
An element $\cP\in \Pot_{2}(kQ)$ is called a {\em potential} on $Q$. It is called {\em reduced} if $\cP \in \Pot_{3}(kQ)$. We say that two potentials are {\em cyclically equivalent} if their difference is in the closure of the span of all elements of the form $ p_{1} \cdots p_{d} - p_{2}\cdots p_{d} p_{1}$ where $p_{1}\cdots p_{d}$ is a cyclic path.

For an arrow $p$ of $Q$ we define $\partial_{p}\from \Pot_{2}(kQ) \to \widehat{kQ}$ as the continuous $k$-linear map taking a cycle $c$ to the sum $\Sigma_{c=xpy} yx$ taken over all decompositions of the cycle $c$ (where $x$ or $y$ are possibly paths of length zero). It is clear that two cyclically equivalent paths have the same image under $\partial_{p}$. We call $\partial_{p}$ the {\em cyclic derivative with respect to} $p$.

 Let $\cP$ be a potential on $Q$ such that no two cyclically equivalent cyclic paths appear in $\cP$. We call the pair $(Q,\cP)$ a {\em quiver with potential} (or QP for short). Associated to a QP we have a {\em Jacobian algebra} defined as
 \[ \cJ (Q,\cP) = \widehat{kQ}/I(\cP), \]
 where $I(\cP)$ is the closure of the ideal generated by all $\partial_{p}\cP$ where $p$ runs over all arrows of $Q$.
 
 For further details we refer the reader to \cite{DWZ}.
 
\subsection{Cluster categories} \label{section.cluster.cats}
The cluster category $\cC$ of a hereditary category $\cH$ was introduced in \cite{BMRRT}. It is defined as the orbit category $\cD/F$, where $F$ is the automorphism $\tau^{-1}[1]$ with $\tau$ the Auslander Reiten translation, and $[1]$ the shift functor of $\cD:=D^b(\cH)$.  
This gives an algebraic triangulated (\cite{Keller}) and Krull-Schmidt (\cite{BMRRT}) category, which is an important tool for studying the tilting theory of $\cH$. It also models the combinatorics of the cluster algebras introduced by Fomin and Zelevinsky (see \cite{FZ1}) in a natural way.

The objects of $\cC$ are the same as in $\cD$. Given objects $X$ and $Y$ in $ \cD$, the space of morphisms $\Hom_{\cC}(X, Y)$ is  defined as $\coprod_{i \in \bbZ} \Hom_{\cD}(F^iX,Y)$. The  morphisms in $\cC$ are thus induced by morphisms and extensions in $\cH$.

The category has an important set of objects, namely the \emph {cluster-tilting objects}. These are maximal rigid objects (maximal with respect to the number of non-isomorphic indecomposable summands). If $T$ is a cluster-tilting object in $\cC$, then $\End_{\cC}(T)$ is called a \emph{cluster-tilted algebra} (\cite{BMR2}).

We shall be particularly interested in the cluster-tilted algebras of finite type. These were characterized by \cite{BMR2} as follows: Let $B=\End_\cC(T)^{\op}$ be a cluster-tilted algebra with $\cC=\cC_H$ the cluster category of some hereditary algebra $H$, and $T$ a tilting object in $\cC$. We then have that $B$ is of finite representation type if and only if $H$ is of finite representation type. In this case $H$ is the path algebra of a Dynkin quiver $Q$, and the underlying graph $\Delta$ of $Q$ is one of $\{\bbA_n,\bbD_m,\bbE_6,\bbE_7,\bbE_8\}$ for $n\ge1$ and $m \ge 4$. In this case,  we say that $\cC$ and $B$ are of type $\Delta$. 

Finite type cluster-tilted algebras are (up to Morita equivalence) determined uniquely by their quiver by \cite{BMR1}. Furthermore, their relations are determined by a potential, which is given by the sum of all minimal cycles of the corresponding quiver (\cite{BMR1})\footnote{The result in \cite{BMR1} used a different equivalent description.}. Recall that a cycle is {\em minimal} if the subquiver generated by the cycle contains only arrows of the cycle and every vertex appears only once.

\subsection{Coverings} We follow \cite{Gabriel}. However our notation varies slightly from \cite{Gabriel}, since our categories have finite direct sums and isomorphic objects which are not equal.

A Krull-Schmidt category $\cC$ is called {\em locally bounded} if for any $X \in \cC$ there are only finitely many isomorphism classes of indecomposables $Y$ such that $(X,Y) \neq 0$, and only finitely many isomorphism classes of indecomposables $Z$ such that $(Z,X) \neq 0$.

A ($k$-linear) functor $F\from \cC \to \cD$  is called a {\em covering functor} if the induced maps
 \[
  \bigoplus \cC(X,Y) \to \cD(FX,A) \text{ and } \bigoplus \cC(Y,X) \to \cD(A,FX)
 \]
are bijective for all $X \in \cC$ and indecomposable $A \in \cD$, where in both cases the sum runs over all isomorphism classes of objects $Y$ such that $FY \simeq A$.

 Let $\cC$ be a locally finite-dimensional category and $G$ a group of ($k$-linear) automorphisms of $\cC$. Assume that the action of $G$ on $\cC$ is {\em free} (that is $gX \not\simeq X$ for each $X\in \cC$ and $1\ne g \in G$) and {\em locally bounded} (for each pair $X,Y$ there are only finitely many $g\in G$ such that $(X,gY)\simeq (g^{-1}X,Y)\ne 0$). Then we have the following proposition.
 
 \begin{proposition}[{\cite[3.1]{Gabriel}}] The quotient $\cC/G$ exists in the category of all locally finite-dimensional $k$-categories, and the canonical projection $\pi \from \cC \to \cC/G$ is a covering functor.
 \end{proposition}
 
 Suppose that we have a covering functor $F\from \cC \to \cD$ such that $Fg=F$ for all $g\in G$. Such a functor induces an isomorphism $\cC/G \simeq \cD$ if and only if $F$ is surjective on the objects and $G$ acts transitively on the fiber $F^{-1}(A)$ for each object $A\in \cD$. If this is the case, we call $F$ a {\em Galois covering}. 
 
 For further details we refer the reader to \cite{BG,Gabriel}.

\subsection{Palu's formula}
\label{section.Palu}
 In this subsection we recall the generalized mutation rule for algebraic 2-CY triangulated categories from \cite{Palu}. This formula is one of the key ingredients in our algorithm for mutating at loops and 2-cycles. We state the result in the following setting: Let $T$ be a cluster-tilting object in the algebraic 2-CY triangulated category $\cT$. Delete the loops and oriented 2-cycles from the quiver of $(T,T)$, and denote the remaining quiver by $Q$. Let $M=(m_{i,j})$ be the skew-symmetric matrix (see Subsection~\ref{section.quiver_mutation}) associated to $Q$. Furthermore, let $T'\in \cT$ be another cluster-tilting object, and define the quiver $Q'$ and the skew-symmetric matrix $M'$ in a likewise manner.

We now approximate $T$ with respect to $T'$. We can write $T=\bigoplus_j T_j$ with $T_j$ indecomposable, and decompose $T'$ similarly. We have that for each $j$ there is a triangle (see \cite{BMRRT,IY}) of the form
\[
   T_j[-1] \to \bigoplus_i \beta_{i,j} T'_j \to \bigoplus_i \alpha_{i,j} T'_i
   \to T_j
\]
and we define the matrix $S=(s_{i,j})$ by setting $s_{i,j}=\alpha_{i,j}-\beta_{i,j}$. Then by \cite[Theorem 12~a)]{Palu} we have that

\begin{equation}\label{formula.Palu}
  M' = S M S^t.
\end{equation}

\subsection{Iyama-Yoshino's reduction}
\label{section.IY} 

 This subsection collects the results we are going to use from \cite{IY}, but restricted to the following setting: Let $\cT$ be a 2-CY triangulated category. Fix a functorially finite subcategory $\cD$ of $\cT$ satisfying $(\cD,\cD[1])=0$. Using $\cD$, we construct the subcategory $\cZ:=\lperp{\cD}[1]$ and the subfactor category $\cU:=\cZ/[\cD]$ of $\cT$. Then the category $\cU$ is triangulated (\cite[Theorem~4.2]{IY}) and $2$-CY (\cite[Theorem~4.9]{IY}). We now describe the shift functor $\shift{1}$ and the standard triangles in $\cU$.

For any $X\in \cZ$, fix a triangle
\[ X \extto{\alpha_X} D_X \extto{\beta_X} X\shift{1}\extto{\gamma_X} X[1] \]
where $D_X \in \cD$, $\alpha_X$ is a left $\cD$-approximation, and define $X\shift{1}$ by this (then $\beta_X$ automatically is a right $\cD$-approximation). The action of $\shift{1}$ on morphisms uses commutative diagrams of triangles like the one above (see \cite[Definition~2.5]{IY}).

Let $X\extto{a}Y \extto{b}Z \extto{c}X[1]$ be a triangle in $\cT$ with $X,Y,Z\in\cZ$. Since $\cT(Z[-1],D_X)=0$ holds, there is a commutative diagram
 \[\scalebox{1}{
 \begin{tikzpicture}[scale=2,yscale=-1]
 \node (X) at (1,1) {$X$};
 \node (Y) at (2,1) {$Y$};
 \node (Z) at (3,1) {$Z$};
 \node (X1) at (4,1) {$X[1]$};
 \node (XX) at (1,2) {$X$};
 \node (DX) at (2,2) {$D_X$};
 \node (Xs1) at (3,2) {$X\shift{1}$};
 \node (XX1) at (4,2) {$X[1]$};
 \draw [->] (X) -- node [auto] {$1_{X}$} (XX);
 \draw [->] (X) -- node [auto] {$a$} (Y);
 \draw [->] (Y) -- node [auto] {$b$} (Z);
 \draw [->] (Z) -- node [auto] {$c$} (X1);
 \draw [->] (XX) -- node [auto] {$\alpha_X$} (DX);
 \draw [->] (DX) -- node [auto] {$\beta_X$} (Xs1);
 \draw [->] (Xs1) -- node [auto] {$\gamma_{X}$} (XX1);
 \draw [->] (Y) -- (DX);
 \draw [->] (Z) -- node [auto] {$d$}(Xs1); 
 \draw [->] (X1) -- node [auto] {$1_{X[1]}$} (XX1); 
 \end{tikzpicture}}
 \]
Now the standard triangles in $\cU$ are the complexes of the form $X\extto{\oa}Y \extto{\ob}Z \extto{\od}X\shift{1}$.

\section{Mutating at oriented cycles}\label{section.mutation.cycles}

 In this section we develop the theory to mutate a cluster-tilting object in several summands (satisfying certain specific conditions) at the same time, in an algebraic $2$-CY triangulated category. To achieve this, we rely heavily on Palu's formula from Section~\ref{section.Palu} and the Iyama-Yoshino construction recalled in Section~\ref{section.IY}. 

\subsection{Exchanging several summands.}

 Throughout this section let $\cT$ be an algebraic $2$-CY triangulated category, and fix $T=T_m \oplus T_f$ a cluster-tilting object in $\cT$, where neither of the summands $T_m$ or $T_f$ is necessarily indecomposable. Let $\cD:= \add T_f$. Then clearly $(\cD,\cD[1])=0$. Define $\cZ$ to be the subcategory $\lperp{\cD}[1]$ and $\cU$ the $2$-CY subfactor category $\cZ/[\cD]$ of $\cT$. By \cite[Theorem 4.9]{IY} we have a one-to-one correspondence between cluster-tilting objects in $\cT$ having $T_f$ as a summand and cluster-tilting objects in $\cU$.

 The main purpose of this section is to mutate the summand $T_m$ of $T$, and leave the remaining part fixed, i.e.\ we want to replace $T_m$ by $T_m'$ in such a way that $T'=T_m' \oplus T_f$ is again a cluster-tilting object in $\cT$. 

In order to do this, we follow the construction explained in Section~\ref{section.IY}. Consider the following triangle in $\cT$
\[
   T_m \extto{a} D \extto{b} T_m\shift{1} \to T_m[1]
\] 
where $a$ (resp.\ $b$)  is a minimal left (resp.\ right) $\cD$-aproximation, $D\in\cD$, and $\shift{1}$ is the shift functor in $\cU$. The object $T_m\shift{1} \oplus T_f$ is cluster-tilting in $\cT$, since $T_m\shift{1}$ is cluster-tilting in $\cU$.

 One could also make the dual construction, by using the following triangle
\[
   T_m[-1] \to T_m \shift{-1} \extto{b'} D' \extto{a'} T_m
\] 
where $a'$ (resp.\ $b'$) is a minimal right (resp.\ left) $\cD$-aproximation and $D'\in\cD$. 

We want to construct the replacement $T_m'$ symmetrically, hence we require $T_m'=T_m\shift{1}$ and $T_m'=T_m\shift{-1}$. Thus we need $T_m\shift{1} \simeq T_m\shift{-1}$.

\begin{construction}\label{cons.mut_several_summands}
Using the notation as above, assume that $T_m\simeq
  T_m\shift{2}$ in $\cU$. Then we define the mutation at the summand $T_m$
  of the cluster-tilting object $T$ in $\cT$ to be 
\[\mu_{T_m}(T)=T_m\shift{1}\oplus T_f\]
\end{construction}

\begin{remark}
The assumption in the construction above is equivalent to requiring that the algebra $\cT(T_m,T_m)/[\cD]$ is self-injective (see \cite{Ringel}).
\end{remark}

\begin{example} \label{example.a9}
 Let $\cT$ be the cluster category of type $\bbA_9$ (This category will be called $\cA_{9,1}$ in Definition~\ref{def.2cy.cats}), and $T=\oplus_{i,j} T_{i_{j}}$ the cluster-tilting object in $\cT$ depicted in the figure below. We want to mutate $T$ at $T_{m}=\oplus_{j} T_{1_{j}}$. Observe that the hypothesis of Construction~\ref{cons.mut_several_summands} is satisfied, and we obtain $T_{m}\shift{1}\simeq\oplus_{j} T'_{1_{j}}$. This process is depicted in the figure below.
 \[
 \scalebox{1}
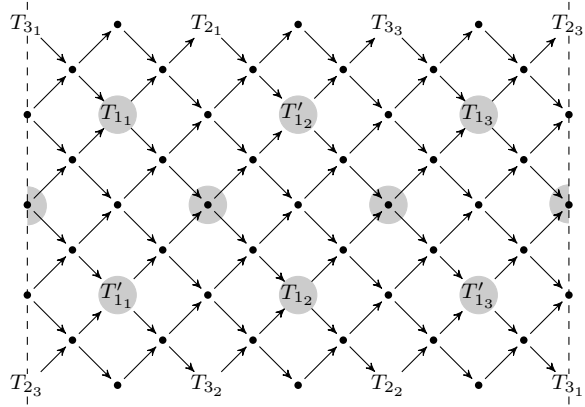
{\begin{tikzpicture}[scale=0.6,yscale=-1]  
  \fill [fill1] (0,4.6) arc (-90:90:12pt);
  \fill [fill1] (12,5.4) arc (90:270:12pt);
 \foreach \x in {0,...,6}
  \foreach \y in {1,3,5,7,9}
   \node (\x-\y) at (\x*2,\y) [vertex] {};
 \foreach \x in {0,...,5}
  \foreach \y in {2,4,6,8}
   \node (\x-\y) at (\x*2+1,\y)  [vertex] {};
 \replacevertex{(0-1)} {[tvertex] {$T_{3_{1}}$}}
 \replacevertex{(2-1)} {[tvertex] {$T_{2_{1}}$}}
 \replacevertex[fill1]{(1-3)} {[tvertex] {$T_{1_{1}}$}}
  \replacevertex[fill1]{(3-7)} {[tvertex] {$T_{1_{2}}$}}
  \replacevertex{(2-9)} {[tvertex] {$T_{3_{2}}$}}
  \replacevertex{(4-9)} {[tvertex] {$T_{2_{2}}$}}
  \replacevertex[fill1]{(5-3)} {[tvertex] {$T_{1_{3}}$}}
  \replacevertex{(4-1)} {[tvertex] {$T_{3_{3}}$}}
  \replacevertex{(0-9)} {[tvertex] {$T_{2_{3}}$}}
 \replacevertex[fill1]{(1-7)} {[tvertex] {$T'_{1_{1}}$}}
 \replacevertex[fill1]{(3-3)} {[tvertex] {$T'_{1_{2}}$}}
 \replacevertex[fill1]{(5-7)} {[tvertex] {$T'_{1_{3}}$}}
 \replacevertex{(6-1)} {[tvertex] {$T_{2_{3}}$}}
 \replacevertex{(6-9)} {[tvertex] {$T_{3_{1}}$}}
  \replacevertex[fill1]{(2-5)} {[vertex] {}}
  \replacevertex[fill1]{(4-5)} {[vertex] {}}
 \foreach \xa/\xb in {0/1,1/2,2/3,3/4,4/5,5/6}
  \foreach \ya/\yb in {1/2,3/2,3/4,5/4,5/6,7/6,7/8,9/8}
   {
    \draw [->] (\xa-\ya) -- (\xa-\yb);
    \draw [->] (\xa-\yb) -- (\xb-\ya);
   }
\draw [dashed] (0,.5) -- (0,0.75); 
\draw [dashed] (0,1.25) -- (0,8.75); 
\draw [dashed] (0,9.25) -- (0,9.5); 
\draw [dashed] (12,.5) -- (12,0.75);
\draw [dashed] (12,1.25) -- (12,8.75);
\draw [dashed] (12,9.25) -- (12,9.5);
 \end{tikzpicture}}
\]
Here, the subfactor category $\cU$ of $\cT$, for $T_{f}=\oplus_{j} (T_{2_{j}}\oplus T_{3_{j}})$, is indicated in light grey, and is easily seen to be equivalent to the cluster category of $\bbA_3$. 
\end{example}

 We now present the setup under which we can mutate at cycles. 

\begin{setup}[for mutation at cycles] \label{setup.cycle} 
Let $\cT$ be an algebraic 2-CY triangulated category, and fix $T=T_m \oplus T_f$ a cluster-tilting object in $\cT$, where neither of the summands $T_m$ or $T_f$ is necessarily indecomposable. Let $\cU$ be the subfactor category of $\cT$ defined by 
\[
 \cU := \lperp{ (\add T_f [1]) }/[\add T_f].
\]
Let $T'=T_m' \oplus T_f$, where the summand $T_m' = T_m\shift{1}$ and the functor $\shift{1}$ is the shift in $\cU$. Furthermore, we have the following assumptions and notation:
\begin{enumerate}
\item No cluster-tilting object in $\cT$ has loops or 2-cycles.
\item Let $Q$ be the quiver of $(T,T)$. Denote by $Q_{m,m}$ the subquiver of $Q$ whose arrows correspond to maps from $T_m$ to $T_m$. Define $Q_{m,f}$, $Q_{f,m}$ and $Q_{f,f}$ likewise. Assume further that the quiver $Q_{m,m}$ is a cycle of length $l\ge 3$. 
\item Denote by $Q'$ the quiver of $(T',T')$ and define $Q'_{a,b}$ as above for all $a,b \in \{m,f\}$.  
\item The relations of $(T,T)$ are given by a potential.
\item The algebra $(T_m,T_m)/[D]$ is the path algebra of $Q_{m,m}$ with minimal relations given by the paths of length $l$.
\end{enumerate}
\end{setup}

 Under this setup, we want to mutate $Q$ at the cycle $Q_{m,m}$. In order to do this, we apply the formula from Equation~\ref{formula.Palu}. Using the same notation as in Section~\ref{section.Palu} we compute as follows. Write
\[
 M= \begin{pmatrix} A & B \\ C & D \end{pmatrix} -\begin{pmatrix} A &
   B \\ C & D \end{pmatrix}^t
  = \begin{pmatrix} A-A^t & B-C^t \\ C-B^t & D-D^t \end{pmatrix}
\] 
where $M_{m,m}:=A$ (resp.\ $M_{m,f}:=B, M_{f,m}:=C,M_{f,f}:=D$) is the matrix of the arrows in $Q_{m,m}$ (resp.\ $Q_{m,f}$,$Q_{f,m}$,$Q_{f,f}$).

  Observe that $A$ is an $l \times l$-matrix that looks like
\[
 A = \begin{pmatrix} 
  0  & 1 & 0       & \cdots &        & 0 \\
     & 0 & 1       & 0      & \cdots & 0 \\
 \vdots    &   & \ddots  & \ddots & \ddots & \vdots \\
     &   &         &        &  \ddots      & 0  \\
   0 &   & \cdots  &        &    0   & 1 \\
   1 & 0 &         & \cdots &        & 0
\end{pmatrix},
\] 
a cyclic permutation matrix. To simplify the notation, we denote all the arrows of $Q_{m,m}$ by $\alpha$, and set $\alpha^i$ to be the composition of $i$ such arrows. We decompose $C= \sum_{i=0}^{l-2} C_i$, where $C_i$ corresponds to the set of arrows 

\begin{align*} 
C_i= \{ \gamma \in Q_{f,m} \, | & \, \gamma \alpha^{i+1} \text{ factors through an arrow in } Q_{f,f}, \\
& \, \gamma \alpha^i \text{ does not factor through an arrow in } Q_{f,f} \}.
\end{align*} 

Dually, we decompose $B = \sum_{i=0}^{l-2} B_i$.  Now, using the fact that the relations come from a potential, we have the following equalities.
\begin{equation}
  C_i^t= A^{i+1} B_i \mbox{ for all } 0 \le i \le l-3.
\end{equation}

 Observe that the matrix $S$ (see Subsection~\ref{section.Palu}) can be written as 
\[ 
S = \begin{pmatrix} 
   -1 & 0 \\ 
   \Sum_{i=0}^{l-2}C_i \Sum_{j=0}^{i}A^j & 1 
\end{pmatrix}.  
\]

 In order to get a more symmetric result, we twist $S$ with the permutation matrix $\left(\begin{smallmatrix} A^{-1} & 0 \\ 0 & 1 \end{smallmatrix}\right)$, and thus we calculate
\[
 M'= \begin{pmatrix} -A^{-1} & 0 \\ \Sum_{i=0}^{l-2}C_i \Sum_{j=0}^{i}A^j & 1 \end{pmatrix}
     \begin{pmatrix} A-A^t & B-C^t \\ C-B^t & D-D^t \end{pmatrix}
     \begin{pmatrix} -A^{-1} & 0 \\ \Sum_{i=0}^{l-2}{C_i} \Sum_{j=0}^{i}A^j & 1 \end{pmatrix}^t.
\]

Then we can prove the following result.
\begin{theorem}[Mutation rule for cycles]
\label{rule.mut_min_cycle}
Suppose that we are in the situation of Setup~\ref{setup.cycle}. Then mutation at $Q_{m,m}$ has the following effect on the quiver:
\begin{enumerate}
\item Arrows in $Q_{m,m}$ remain unchanged.
\item Arrows in $Q_{f,m}$: Any arrow $\gamma$ in $C_i$, with $1\le i \le l-3$ remains unchanged. For an arrow $\gamma$ in $C_{l-2}$, we consider the path $\gamma \alpha^{l-1}$ going from $a$ to $b$. Then replace $\gamma$ by an arrow $[\gamma\alpha^{l-1}]^t$ going from $b$ to $a$. 
\item Arrows in $Q_{m,f}$: Apply the dual process for the arrows in $Q_{m,f}$.
\item Arrows in $Q_{f,f}$ remain unchanged. 
\item Furthermore, add an arrow $[\gamma \alpha^i \beta]$ for each composition $\gamma \alpha^i \beta$ where $\gamma\in Q_{f,m}$, $\alpha \in Q_{m,m}$, $\beta\in Q_{m,f}$ such that
\begin{enumerate}
  \item neither $\gamma \alpha^i$ nor $\alpha^i \beta$ factor through an arrow in $Q_{f,f}$,
  \item  $\gamma \in C_{l-2}$ or $\beta \in B_{l-2}$, i.e.\ at least one of $\gamma$ or $\beta$ has no extra relations with the cycle of $Q_{m,m}$.
\end{enumerate}
\item Finally, remove any loops or 2-cycles from the mutated quiver.
\end{enumerate}
\end{theorem}

\begin{proof}
The calculations are straight-forward but somewhat lengthy. See Appendix~\ref{appendix.ugly}.
\end{proof}

\begin{example} \label{example.quivers.a9}
 Let $T$ in $\cT$ be as in Example~\ref{example.a9}. We depict the quiver $Q$ of $(T,T)$ to the left in the diagram below:
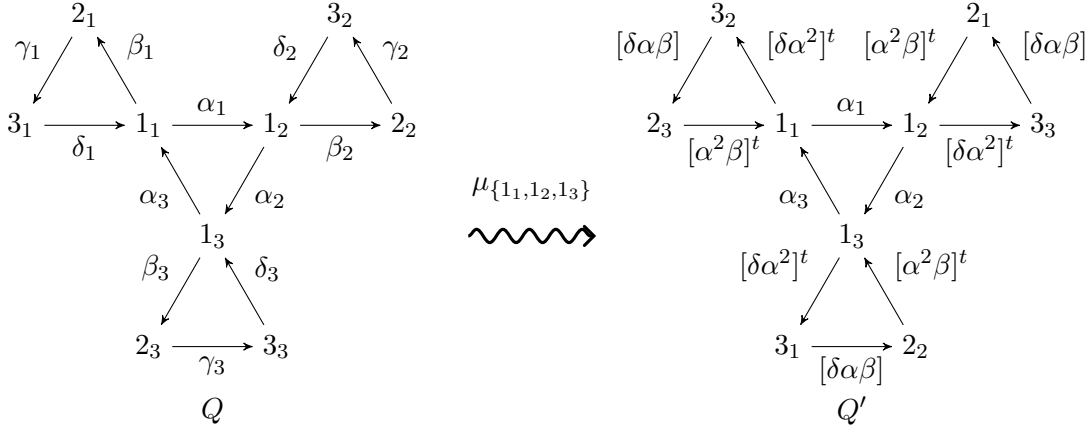
\[
\scalebox{1}{ \begin{tikzpicture}[scale=0.85,yscale=-1]
 \node at (4,7.464) {$Q$};
 \node (M1) at (3,3)  {$1_{1}$};
 \node (M2) at (5,3)  {$1_{2}$};
 \node (M3) at (4,4.732)  {$1_{3}$};
 \node (F13) at (3,6.464)  {$2_{3}$};
 \node (F23) at (5,6.464)  {$3_{3}$};
 \node (F12) at (7,3)  {$2_{2}$};
 \node (F22) at (6,1.276)  {$3_{2}$};
 \node (F11) at (2,1.276)  {$2_{1}$};
 \node (F21) at (1,3)  {$3_{1}$};
 \draw [->] (M1) -- node [auto] {$\alpha_{1}$} (M2);
 \draw [->] (M2) -- node [auto] {$\alpha_{2}$} (M3);
 \draw [->] (M3) -- node [auto] {$\alpha_{3}$} (M1);
 \draw [->] (M1) -- node [auto,swap] {$\beta_{1}$} (F11);
 \draw [->] (F11) -- node [auto,swap] {$\gamma_{1}$} (F21);
 \draw [->] (F21) -- node [auto,swap] {$\delta_{1}$} (M1);
 \draw [->] (M2) -- node [auto,swap] {$\beta_{2}$} (F12);
 \draw [->] (F12) -- node [auto,swap] {$\gamma_{2}$} (F22);
 \draw [->] (F22) -- node [auto,swap] {$\delta_{2}$} (M2);
 \draw [->] (M3) -- node [auto,swap] {$\beta_{3}$} (F13);
 \draw [->] (F13) -- node [auto,swap] {$\gamma_{3}$} (F23);
 \draw [->] (F23) -- node [auto,swap] {$\delta_{3}$} (M3);
 \draw [leadsto] (8,4.732) -- (10,4.732);
 \node at (9,4) {$\mu_{ \left\{1_1,1_2,1_3\right\} }$};
  \pgftransformshift{\pgfpoint{10cm}{0cm}};
 \node at (4,7.464) {$Q'$};
   \node (M1) at (3,3)  {$1_{1}$};
 \node (M2) at (5,3)  {$1_{2}$};
 \node (M3) at (4,4.732)  {$1_{3}$};
 \node (F13) at (3,6.464)  {$3_{1}$};
 \node (F23) at (5,6.464)  {$2_{2}$};
 \node (F12) at (7,3)  {$3_{3}$};
 \node (F22) at (6,1.276)  {$2_{1}$};
 \node (F11) at (2,1.276)  {$3_{2}$};
 \node (F21) at (1,3)  {$2_{3}$};
 \draw [->] (M1) -- node [auto] {$\alpha_{1}$} (M2);
 \draw [->] (M2) -- node [auto] {$\alpha_{2}$} (M3);
 \draw [->] (M3) -- node [auto] {$\alpha_{3}$} (M1);
 \draw [->] (M1) -- node [auto,swap] {$[\delta\alpha^2]^t$} (F11);
 \draw [->] (F11) -- node [auto,swap] {$[\delta\alpha\beta]$} (F21);
 \draw [->] (F21) -- node [auto,swap] {$[\alpha^2\beta]^t$} (M1);
 \draw [->] (M2) -- node [auto,swap] {$[\delta\alpha^2]^t$} (F12);
 \draw [->] (F12) -- node [auto,swap] {$[\delta\alpha\beta]$} (F22);
 \draw [->] (F22) -- node [auto,swap] {$[\alpha^2\beta]^t$} (M2);
 \draw [->] (M3) -- node [auto,swap] {$[\delta\alpha^2]^t$} (F13);
 \draw [->] (F13) -- node [auto,swap] {$[\delta\alpha\beta]$} (F23);
 \draw [->] (F23) -- node [auto,swap] {$[\alpha^2\beta]^t$} (M3);
 \end{tikzpicture}}
\]
where the potential is given by the sum of all the minimal cycles. It is clear that Setup~\ref{setup.cycle} is satisfied. We want to mutate this quiver at the minimal cycle spanned by the set of vertices $\{1_{1},1_{2},1_{3}\}$. Using the same notation as in the mutation rule, we observe that we have decompositions $C=C_{0}+C_{1}$ and $B=B_{0}+B_{1}$ since the length of the minimal $\alpha$-cycle is $3$. It is easy to see that $C_{0}$ and $B_{0}$ are empty, $B_{1}=\{\beta_{i} \, | \, i=1,2,3 \}$ and $C_{1}=\{\delta_{i} \, | \, i=1,2,3 \}$.  We are ready to apply the rule:
\begin{enumerate}
 \item{$Q'_{m,m}$.} The $\alpha$-arrows stay the same.
 \item{$Q'_{m,f}$.} For each path $\alpha^2\beta$ we add an arrow $[\alpha^2\beta]^t$ going in the opposite direction.
 \item{$Q'_{f,m}$.} For each path $\delta \alpha^2$ we add an arrow $[\delta\alpha^2]^t$ going in the opposite direction.
 \item{$Q'_{f,f}$.} The $\gamma$-arrows stay the same. 
 \item Furthermore, we add all the compositions $[\delta \beta]$ and $[\delta \alpha \beta]$. Then we end up with a quiver like in the figure shown below. Here we have indicated in which step the arrows have been added.

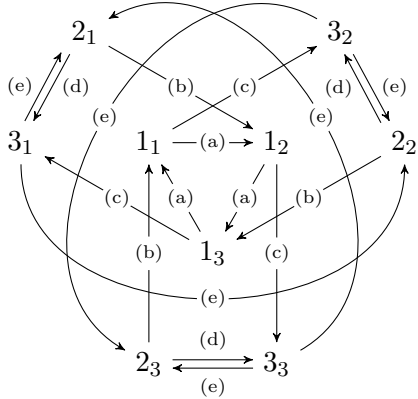
\[
\scalebox{1}{ \begin{tikzpicture}[scale=0.85,yscale=-1]
 \node (M1) at (3,3)  {$1_{1}$};
 \node (M2) at (5,3)  {$1_{2}$};
 \node (M3) at (4,4.732)  {$1_{3}$};
 \node (F13) at (3,6.464)  {$2_{3}$};
 \node (F23) at (5,6.464)  {$3_{3}$};
 \node (F12) at (7,3)  {$2_{2}$};
 \node (F22) at (6,1.276)  {$3_{2}$};
 \node (F11) at (2,1.276)  {$2_{1}$};
 \node (F21) at (1,3)  {$3_{1}$};
 \draw [->] (M1) -- node [WhiteBK,pos=0.5] {\tiny (a)} (M2);
 \draw [->] (M2) -- node [WhiteBK,pos=0.5] {\tiny (a)} (M3);
 \draw [->] (M3) -- node [WhiteBK,pos=0.5] {\tiny (a)} (M1);
 \draw [->] (F11.235) -- node [right] {\tiny (d)} (F21.60);
 \draw [<-] (F11.220) -- node [left] {\tiny (e)} (F21.80);
 \draw [<-] (F12.125) -- node [right] {\tiny (e)} (F22.300);
 \draw [->] (F12.140) -- node [left] {\tiny (d)} (F22.280);
 \draw [->] (F13.10) -- node [above] {\tiny (d)} (F23.170);
 \draw [<-] (F13.350) -- node [below] {\tiny (e)} (F23.190);
 \draw [->] (F13) -- node [WhiteBK,pos=0.5] {\tiny (b)} (M1);
 \draw [->] (F11) -- node [WhiteBK,pos=0.5] {\tiny (b)} (M2);
 \draw [->] (F12) -- node [WhiteBK,pos=0.5] {\tiny (b)} (M3);
 \draw [->] (M3) -- node [WhiteBK,pos=0.5] {\tiny (c)} (F21);
 \draw [->] (M2) -- node [WhiteBK,pos=0.5] {\tiny (c)} (F23);
 \draw [->] (M1) -- node [WhiteBK,pos=0.5] {\tiny (c)} (F22);
 \draw [->,bend left=90,looseness=1.2] (F21) to node [WhiteBK,pos=0.5] {\tiny (e)} (F12);
 \draw [->,bend left=90,looseness=1.2] (F22) to node [WhiteBK,pos=0.5] {\tiny (e)} (F13);
 \draw [->,bend left=90,looseness=1.2] (F23) to node [WhiteBK,pos=0.5] {\tiny (e)} (F11);
 \end{tikzpicture}}
\]

 \item Now eliminate all loops and $2$-cycles.
\end{enumerate}

Then one obtains the quiver $Q'$ to the right in the figure at the beginning of the example, which is the quiver of the endomorphism ring of $\mu_{\{1_{1},1_{2},1_{3}\}}(T)$ as one can also see from the AR-quiver of $\cT$ in Example~\ref{example.a9}. 
\end{example}

 \begin{example} Consider the canonical cluster-tilting object $T$ from the stable module category of the preprojective algebra  of $A_6$. The quiver $Q$ of $(T,T)$ is depicted to the left in the figure below.
 \[ \scalebox{1}
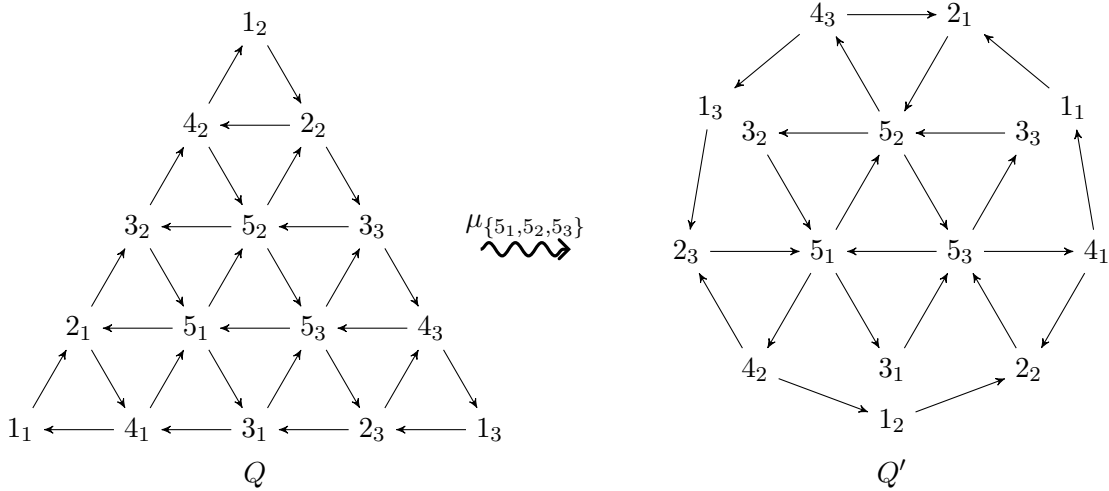
{
 \begin{tikzpicture}[scale=0.6]
   \node at (0,-4) {$Q$};
   \node (5-1) at (-1.3,-0.75) {$5_1$};
   \node (5-2) at (0,1.5) {$5_2$};
   \node (5-3) at (1.3,-0.75) {$5_3$};
   \node (3-1) at (0,-3) {$3_1$};
   \node (3-2) at (-2.6,1.5) {$3_2$};
   \node (3-3) at (2.6,1.5) {$3_3$};
   \node (4-1) at (-2.6,-3) {$4_1$};
   \node (4-2) at (-1.3,3.75) {$4_2$};
   \node (4-3) at (3.9,-0.75) {$4_3$};
   \node (2-1) at (-3.9,-0.75) {$2_1$};
   \node (2-2) at (1.3,3.75) {$2_2$};
   \node (2-3) at (2.6,-3) {$2_3$};
   \node (1-1) at (-5.2,-3) {$1_1$};
   \node (1-2) at (0,6) {$1_2$};
   \node (1-3) at (5.2,-3) {$1_3$};
   \foreach \y in {1,2,3} {
    \draw [->] (5-\y)--(3-\y);
    \draw [->] (3-\y)--(4-\y);
    \draw [->] (4-\y)--(5-\y);
    \draw [->] (5-\y)--(2-\y);   
    \draw [->] (2-\y)--(4-\y);
    \draw [->] (4-\y)--(1-\y);
    \draw [->] (1-\y)--(2-\y);
    }
   \foreach \x/\y in {2/1,3/2,1/3}{
    \draw [->] (2-\y)--(3-\x);
    \draw [->] (3-\x)--(5-\y);
    \draw [->] (5-\y)--(5-\x);
   }
 \node at (6,1.5) {$\mu_{\left\{5_1,5_2,5_3\right\}}$};
 \draw [leadsto] (5,1) -- (7,1);
 \end{tikzpicture}
\quad \quad
 \begin{tikzpicture}[scale=0.7]
   \node at (0, -4.97) {$Q'$};
   \node (5-1) at (-1.3,-0.75) {$5_1$};
   \node (5-2) at (0,1.5) {$5_2$};
   \node (5-3) at (1.3,-0.75) {$5_3$};
   \node (3-1) at (0,-3) {$3_1$};
   \node (3-2) at (-2.6,1.5) {$3_2$};
   \node (3-3) at (2.6,1.5) {$3_3$};
   \node (4-1) at (-2.6,-3) {$4_2$};
   \node (4-2) at (-1.3,3.75) {$4_3$};
   \node (4-3) at (3.9,-0.75) {$4_1$};
   \node (2-1) at (-3.9,-0.75) {$2_3$}; 
   \node (2-2) at (1.3,3.75) {$2_1$}; 
   \node (2-3) at (2.6,-3) {$2_2$}; 
   \node (1-1) at (-3.44,1.98) {$1_3$};
   \node (1-2) at (3.44,1.98) {$1_1$}; 
   \node (1-3) at (0,-3.97) {$1_2$}; 
   \foreach \y in {1,2,3} {
    \draw [->] (5-\y)--(3-\y);
    \draw [->] (5-\y)--(4-\y);
    \draw [->] (2-\y)--(5-\y);   
    \draw [->] (4-\y)--(2-\y);
    \draw [->] (1-\y)--(2-\y);
    }
   \foreach \x/\y in {2/1,3/2,1/3}{
    \draw [->] (3-\x)--(5-\y);
    \draw [->] (4-\x)--(1-\y);
    \draw [->] (5-\y)--(5-\x);
   }
 \end{tikzpicture}}
 \]
where the potential is given by the sum of all the minimal triangles. It is not hard to see that Setup~\ref{setup.cycle} is satisfied (to check the no loops and 2-cycles condition see \cite{GLS,BIRSc}). Let $Q_{m,m}=\{5\to 5\}$, $Q_{f,m}=\{3\to 5, 4\to 5\}$, $Q_{m,f}=\{5\to 2, 5\to 3\}$ and $Q_{f,f}$ the rest. We want to mutate at the minimal cycle of length $3$ given by $Q_{m,m}$. Decompose $C=C_0+C_1$ where $C_0=\{3\to 5\}$ and $C_1=\{4\to 5\}$. Similarly, $B = B_0+B_1$ where $B_0=\{5\to 3\}$ and $B_1=\{5\to 2\}$. We apply the mutation rule for cycles: 
\begin{enumerate}
 \item{$5\to 5$.} These arrows stay the same.
 \item{$5\to \{3,2\}$.} For each path $5\to 5\to 2$ we add an arrow $[2\to 5]$. For the paths $5\to 5\to 3$, we do not add arrows $[3\to 5]$ since $5\to 5\to 3 = 5\to 2 \to 3$, factoring through the arrows $2\to 3$ in $Q_{f,f}$.
 \item{$\{4,3\} \to 5$.} For each path $4\to 5 \to 5$ we add an arrow $[5\to 4]$. For the paths $3\to 5\to 5$, we do not add arrows $[5\to 3]$ since $3\to 5\to 5 = 3\to 4 \to 5$, factoring through the arrows $3\to 4$ in $Q_{f,f}$.
 \item The remaining arrows stay the same. 
 \item Furthermore, we add arrows for all the compositions $[3\to 5 \to 2]$, $[4\to 5 \to 3]$, $[4\to 5 \to 2]$ and $[4\to 5\to 5\to 2]$. 
 \item Now eliminate all loops and $2$-cycles.
\end{enumerate}
 Then one obtains the quiver $Q'$ to the right in the figure above. 
 \end{example}

\subsection{Mutation rule at cycles vs.\ FZ-mutation}

In both examples above one can check that the result of mutating in a cycle can also be obtained by a sequence of FZ-mutations in the vertices of the cycle (using vertices multiple times, not just every vertex once). In this section we will show that this is no coincidence, but that there always exists a sequence of FZ-mutations that corresponds to our mutation rule at cycles.

 Let the cluster-tilting objects $T=T_m\oplus T_f$ and $T'=T_m\shift{1}\oplus T_f$ in $\cT$, the category $\cU$, and the quivers $Q$ and $Q'$ be as in Setup~\ref{setup.cycle}. Assume the conditions of Setup~\ref{setup.cycle} hold. Using the same notation, we observe that the cluster-tilting object $T_m$ in $\cU$ has  $\Gamma:=(T_m,T_m)/[\add T_f]$ as endomorphism ring, which is a cluster-tilted algebra of type $\bbD_l$. It is not hard to find a sequence of mutations taking us from $Q_\Gamma$, the quiver of $\Gamma$, to a hereditary quiver $Q_H$. Then, by using \cite[Main Theorem]{KR}, we see that the 2-CY category $\cU$ is triangle equivalent to the cluster category of $\bbD_l$. Using the methods developed in \cite{BOW1}, we observe that we can choose $T_m$ as indicated in dark gray in the Auslander-Reiten quiver of $\cU$ (where $l$ is assumed to be odd) shown below. Here $T_m\shift{1}$ is indicated in lighter gray.
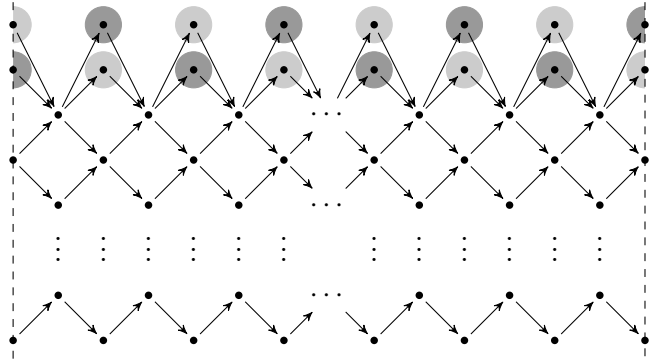
\[ \scalebox{1}{ 
\begin{tikzpicture}[scale=0.6,yscale=-1]
 \foreach \x in {2,4,6}
  \fill [fill1] (\x*2,1) circle (.4);
 \foreach \x in {1,3,5}
  \fill [fill1] (\x*2,2) circle (.4);
 \fill [fill1] (0,.6) arc (-90:90:.4);
 \fill [fill1] (14,2.4) arc (90:270:.4);
 \foreach \x in {2,4,6}
  \fill [fill2] (\x*2,2) circle (.4);
 \foreach \x in {1,3,5}
  \fill [fill2] (\x*2,1) circle (.4);
 \fill [fill2] (0,1.6) arc (-90:90:.4);
 \fill [fill2] (14,1.4) arc (90:270:.4);
 \foreach \x in {0,...,7}
  \foreach \y in {1,2,4,8}
   \node (\y-\x) at (\x*2,\y) [vertex] {};
 \foreach \x in {0,...,6}
  \foreach \y in {3,5,7}
   \node (\y-\x) at (\x*2+1,\y) [vertex] {};
 \replacevertex{(3-3)}{{$\cdots$}}
 \replacevertex{(5-3)}{{$\cdots$}}
 \replacevertex{(7-3)}{{$\cdots$}}
 \foreach \xa/\xb in {0/1,1/2,2/3,3/4,4/5,5/6,6/7}
  \foreach \ya/\yb in {1/3,2/3,4/3,4/5,8/7}
   {
    \draw [->] (\ya-\xa) -- (\yb-\xa);
    \draw [->] (\yb-\xa) -- (\ya-\xb);
   }
 \foreach \x in {1,2,3,4,5,6,8,9,10,11,12,13}
  \node at (\x,5.8) {$\vdots$};
 \draw [dashed] (0,.5) -- (1-0.north);
 \draw [dashed] (1-0.south) -- (0,8.5); 
 \draw [dashed] (14,.5) -- (2-7.north);
 \draw [dashed] (2-7.south) -- (14,8.5);
\end{tikzpicture} } \]
Since $\cU$ is mutation connected, we can always find a sequence of (cluster-tilting) mutations taking us from $T_m$ to $T_m\shift{1}$. Applying this sequence of mutations in $\cT$, we obtain a sequence of (cluster-tilting) mutations taking us from $T$ to $T'$. Since $\cT$ has no loops or 2-cycles by assumption, cluster-tilting mutation corresponds to FZ-mutation at the level of quivers (see \cite{BIRSc}). Therefore, applying the corresponding sequence of FZ-mutations to the quiver $Q$, we obtain the quiver $Q'$. Hence we have proved the following result.

\begin{theorem}\label{thm.mut_cycles.FZmut} Let $T=T_m \oplus T_f$ be a cluster tilting object in $\cT$, such that the conditions of Setup~\ref{setup.cycle} hold. Let $T'=T_m\shift{1}\oplus T_f$ be the cluster-tilting object obtained after mutating $T$ at $T_m$ as in Construction~\ref{cons.mut_several_summands}. Denote by $Q$ (resp.\ $Q'$) the quiver of the 2-CY tilted algebra $(T,T)$ (resp.\ $(T',T')$). Let $Q_{m,m}$ denote the quiver of $(T_m,T_m)$. Then there exists a sequence of FZ-mutations taking us from $Q$ to $Q'$, which correspond to mutating at $Q_{m,m}$ as in Theorem~\ref{rule.mut_min_cycle}.
\end{theorem}

\begin{corollary}
Mutating in cycles in algebraic 2-CY triangulated categories, the cluster tilting objects remain in the same mutation component.
\end{corollary}

In view of Theorem~\ref{thm.mut_cycles.FZmut} and \cite[Proposition 5.1]{BIRSm}, we have the following improvement of Setup~\ref{setup.cycle}. Suppose that the 2-CY tilted algebra $(T,T)$ is isomorphic to $\cJ(Q,\cP)$, the Jacobian algebra of a QP $(Q,\cP)$ (see Section~\ref{sec.QP}), where the quiver of $(Q,\cP)$ has no loops or 2-cycles. Let $\mu_{v_r} \cdots \mu_{v_1}$ be the sequence of FZ-mutations at the vertices $v_1,\ldots,v_r$ taking us from $Q_1:=Q,$ to $Q_r:=Q'$. Then we can relax condition~(a) of Setup~\ref{setup.cycle} to the following. 
\begin{enumerate} 
\item[(a')] No 2-cycles start in vertex $v_{i+1}$ in the quiver of $\mu_{v_i} \cdots \mu_{v_1} (Q_1,\cP_1)$ for $1 \le i \le r-1$. Here the mutations are of quivers with potential, where $(Q_1,\cP_1):= (Q,\cP)$ and $(Q_{i+1},\cP_{i+1}):= \mu_{v_i}(Q_i,\cP_i)$ for $1 \le i \le r-1$. 
\end{enumerate}

This assures that no loops or 2-cycles appear at each step, and that 
\[
(\mu_{v_i}\cdots\mu_{v_1}(T),\mu_{v_i}\cdots\mu_{v_1}(T)) \simeq \cJ(Q_{i+1},\cP_{i+1}) \text{ for } 1 \le i \le r-1 .
\] 

\section{Mutating at loops and 2-cycles}\label{section.mut.loops.2cycles}

 In this section, we build on the theory of Section~\ref{section.mutation.cycles}, in order to develop an algorithm to mutate the quivers of cluster-tilting objects in algebraic $2$-CY triangulated categories, having loops and $2$-cycles. This is done under certain restrictions, as we now explain.

Let $\cT$ be an algebraic 2-CY triangulated category such that no cluster-tilting object in $\cT$ has loops and/or 2-cycles. Suppose that $\pi \from \cT \to \ocT$ is a Galois covering of algebraic triangulated categories (then $\ocT$ automatically also is $2$-CY). Denote by $G$ the group of $k$-linear automorphisms of $\cT$ such that $\ocT = \cT/G$. For an object $X\in\cT$, we denote by $\oX=\pi(X)\in\ocT$. For an object $Y\in\ocT$ we denote by $\laY=\pi^{-1}(Y)\in \cT$. 

We now show that we can lift cluster-tilting objects from $\ocT$ to $\cT$.

\begin{proposition}\label{prop.lift.cto} Let $\dT$ be a cluster-tilting object in $\ocT$. The orbit $\uT = \downstairs{ \laT} $ is a cluster-tilting object in $\cT$.
\end{proposition}
\begin{proof}  

Using the bijections of the Hom-spaces between $\cT$ and $\ocT$ given by $\pi$, we have that
\[
 0 = \Ext^1_{\ocT}(\dT ,\upstairs{\oT}) \simeq \Ext^1_{\cT}(\uT, \uT )
\]   
and thus $\uT$ has no self extensions. Now assume that for $X$ in $\cT$, we have that $\Ext^1_{\cT}(\uT,X)=0$. This implies that $\Ext^1_{\ocT}(\dT,\oX)=0$, and thus $\oX$ is in $\add \dT$. But this just means that $\overleftarrow{\oX}$ is in $\add \uT$. In particular, $X$ is in $\add \uT$.
\end{proof}

 For the rest of the section, fix a basic cluster-tilting object $\dT = \dT_m \oplus \dT_f$ in $\ocT$, where the summand $\dT_m$ is indecomposable. Define $\ocD:= \add \dT_f$ and $\ocZ:=\lperp{\ocD}[1]$. Then the subfactor category $\ocU:=\ocZ/[\ocD]$ is a 2-CY triangulated category (see Subsection~\ref{section.IY}).  

 Write $\uT =\uT_m \oplus \uT_f$ for the lift of $\dT$ in $\cT$, such that $\uT_m = \downstairs{\laT}_{\! m}$ and $\uT_f = \downstairs{\laT}_{\! f}$. As in the previous paragraph, let $\cD:= \add \uT_f$ and define $\cZ$ to be the subcategory $\lperp{\cD}[1]$ of $\cT$. 
Again, we have that the subfactor category $\cU:=\cZ/[\cD]$ forms a 2-CY triangulated category. Then we have the following.

\begin{proposition} \label{prop.covering}
 The covering functor $\pi\from \cT \to \ocT$ induces a triangle functor $\wpi \from \cU \to \ocU$ that is also a Galois covering.
\end{proposition}

\begin{proof}
 First, observe that $\pi$ sends $\cD$ onto $\ocD$ and $\cZ$ onto $\ocZ$. Thus we have a well defined functor $\wpi \from \cU \to \ocU$. Second, note that $\pi$ sends $\cD$-approximations to $\ocD$-approximations, since $\uT_{f}=\downstairs{\laT}_f$. Thus we have that $\wpi\circ \shift{1}_\cU=\shift{1}_\ocU \circ\wpi$, where $\shift{1}_{\cU}$ and $\shift{1}_{\ocU}$ denote the shift functors in the categories $\cU$ and $\ocU$, respectively.
 
 Now to see that $\wpi$ is a triangle functor, let $X\to Y \to Z \to X \shift{1}$ be a standard triangle in $\cU$. This triangle comes from the following commutative diagram of triangles in $\cT$
 \[\scalebox{1}{
 \begin{tikzpicture}[scale=2,yscale=-1]
 \node (X) at (1,1) {$X$};
 \node (Y) at (2,1) {$Y$};
 \node (Z) at (3,1) {$Z$};
 \node (X1) at (4,1) {$X[1]$};
 \node (XX) at (1,2) {$X$};
 \node (DX) at (2,2) {$D_X$};
 \node (Xs1) at (3,2) {$X\shift{1}$};
 \node (XX1) at (4,2) {$X[1]$};
 \draw [->] (X) -- node [auto] {$1_{X}$} (XX);
 \draw [->] (X) --  (Y);
 \draw [->] (Y) --  (Z);
 \draw [->] (Z) --  (X1);
 \draw [->] (XX) -- node [auto] {$\alpha_X$} (DX);
 \draw [->] (DX) -- node [auto] {$\beta_X$} (Xs1);
 \draw [->] (Xs1) --  (XX1);
 \draw [->] (Y) -- (DX);
 \draw [->] (Z) -- (Xs1); 
 \draw [->] (X1) -- node [auto] {$1_{X[1]}$} (XX1); 
 \end{tikzpicture}}
 \]
where the morphism $\alpha_X$ (resp.\ $\beta_X$) is a left (resp.\ right) $\cD$-approximation and $D_X$ belongs to $\cD$ (see Subsection~\ref{section.IY}). This diagram descends via $\pi$ to the following commutative diagram of triangles in $\ocU$
 \[ \scalebox{1}{
 \begin{tikzpicture}[scale=2,yscale=-1]
 \node (X) at (1,1) {$\oX$};
 \node (Y) at (2,1) {$\oY$};
 \node (Z) at (3,1) {$\oZ$};
 \node (X1) at (4,1) {$\oX[1]$};
 \node (XX) at (1,2) {$\oX$};
 \node (T) at (2,2) {$\oD_X$};
 \node (Xs1) at (3,2) {$\oX\shift{1}$};
 \node (XX1) at (4,2) {$\oX[1]$};
 \draw [->] (X) -- node [auto] {$1_{\oX}$} (XX);
 \draw [->] (X) -- (Y);
 \draw [->] (Y) -- (Z);
 \draw [->] (Z) -- (X1);
 \draw [->] (XX) -- node [auto] {$\pi(\alpha_X)$} (T);
 \draw [->] (T) -- node [auto] {$\pi(\beta_X)$} (Xs1);
 \draw [->] (Xs1) -- (XX1);
 \draw [->] (Y) -- (T);
 \draw [->] (Z) -- (Xs1); 
 \draw [->] (X1) -- (XX1); 
 \end{tikzpicture}}
 \]
 where the morphism $\pi(\alpha_X)$ (resp.\ $\pi(\beta_X)$) is a left (resp.\ right) $\ocD$-approximation and $\oD_X$ belongs to $\ocD$. Hence we obtain a standard triangle $\oX\to \oY \to \oZ \to \oX \shift{1}$ in $\ocU$. It is not difficult to see that $G$ acts freely and transitively on the fibers of $\wpi$. Thus $\wpi$ is a Galois covering.
\end{proof}

We now define the setup under which we can mutate at loops and $2$-cyles.

\begin{setup}[for mutation at loops and $2$-cycles] \label{setup.muation_loops} 
Let $\pi \from \cT \to \ocT$ be a Galois covering of algebraic 2-CY triangulated categories. Assume we have a (basic) cluster-tilting object $\dT = \dT_m \oplus \dT_f$ in $\ocT$, where the summand $\dT_m$ is indecomposable. We write $\uT_m=\downstairs{\laT}_{\! m}$, $\uT_f=\downstairs{\laT}_{\! f}$, and $\uT=\downstairs{\laT}$ ($= \uT_m \oplus \uT_f$). Furthermore, we have the following assumptions and notation:
\begin{enumerate}
 \item The category $\cT$ and the cluster-tilting objects $\uT =\uT_m \oplus \uT_f$ and $\uT' =\uT_m\shift{1}_{\cU} \oplus \uT_f$ in $\cT$, are as in Setup~\ref{setup.cycle}.
 \item Denote by $\oQ$ the quiver of $(\dT,\dT)$. The functor $\pi\from \cT \to \ocT$ induces a covering morphisms of quivers, which we also denote by $\pi\from Q \to \oQ$.
 \item Let $\oQ_{m,m}$ denote the quiver of $(\dT_m,\dT_m)/[\dT_f]$. Then $\oQ_{m,m}$ has a single vertex, possibly with a loop and/or $2$-cycles adjacent to this vertex in $\oQ$.
\end{enumerate}
\end{setup}

\begin{remark}  
Observe that under the setup above, the quiver $Q_{m,m}$ is a (possibly disjoint union of ) cycle(s) of length $l \ge 3$ in $Q$ with minimal relations given by the paths of length $l$. 
\end{remark}

 We now present the main theorem of this section.

\begin{theorem}[Mutation at loops and 2-cycles]\label{theorem.mutation.loops} Let $\pi\from \cT \to \ocT$ be a Galois covering of algebraic $2$-CY triangulated categories. Let $\dT=\dT_m\oplus \dT_f$ be a (basic) cluster-tilting object in $\ocT$ with $\dT_m$ indecomposable and denote by $\dT'_m$ the other complement of the almost complete cluster-tilting object $\dT_f$. Write $\uT=\uT_m \oplus \uT_f$ for the lift of $\dT$ to $\cT$ via $\pi$. Suppose that the conditions of Setup~\ref{setup.muation_loops} are satisfied. Then, using the same notation as in Setup~\ref{setup.muation_loops}, we have that $\dT'=\dT'_m \oplus \dT_f = \pi(\uT')$. Furthermore, if $\oQ$ denotes the quiver of $(\dT',\dT')$, then $\oQ'=\pi(Q')$.
\end{theorem}

\begin{proof}
From the proof of Proposition~\ref{prop.covering}, we see that mutating in $\ocT$ corresponds to mutating in $\cT$ as in Construction~\ref{cons.mut_several_summands}.

If the quiver of $\dT$ has loops or 2-cycles, these disappear when lifting to $\uT$ since by assumption cluster-tilting objects in $\cT$ have no loops or $2$-cycles. This allows us to mutate via the cover as follows: We mutate $Q$ at $Q_{m,m}$ using Theorem~\ref{rule.mut_min_cycle} to obtain $Q'$. Then $\oQ'$, the quiver of $\dT'$, is given by $\pi(Q')$.

Let us consider the case when $\oQ_{m,m}$ does not have a loop, but there are possibly $2$-cycles adjacent to the only vertex in $\oQ_{m,m}$. Then it is not difficult to see that the subfactor category $\cU$ of $\cT$ is just the product of cluster categories of type $A_1$. This means that the quiver $Q_{m,m}$ is a disjoint union of isolated vertices, and FZ-quiver mutation at all the vertices in $Q_{m,m}$ (in any order) in the cover gives the correct answer. Then project back using $\pi$.\qedhere
\end{proof}

\begin{remark} When $\cT$ is a cluster category, we know that the relations of any cluster-tilted algebra are determined by its quiver (\cite[5.11]{BIRSm}). In this case, the relations of the $2$-CY tilted algebra $(\dT,\dT)$ are uniquely determined by the cover. 
\end{remark}

We now present an example that illustrates the procedure above.

\begin{example}\label{ex.mut.loop.d6}
Let $\cT = \cD_{6,1}$ be the cluster category of $\bbD_6$ and $\ocT=\cD_{6,3}$ the covering of order $3$ (see the diagram below -- also see Definition~\ref{def.2cy.cats} for notation). It is not hard to see that we have a Galois covering $\pi\from \cT \to \ocT$. Let $\dT=T_1 \oplus T_2$ be the cluster-tilting object in $\ocT$ and denote its lift to $\cT$ by $\uT=\oplus_{i,j} T_{i_j}$ for $1\le i \le 2$ and $1\le j \le 3$, both shown in the figure below. Similarly, let $\dT'=T'_1\oplus T_2$ be the cluster-tilting object obtained by replacing $T_1$ in $\dT$. We can also obtain $\dT'$ by replacing $\oplus_{j=1}^3 T_{1_j}$ by $\oplus_{j=1}^3 T_{1_j}^*$ in $\uT$ and then projecting to $\ocT$ via $\pi$, as illustrated in the figure below.
 \[
 \scalebox{1}
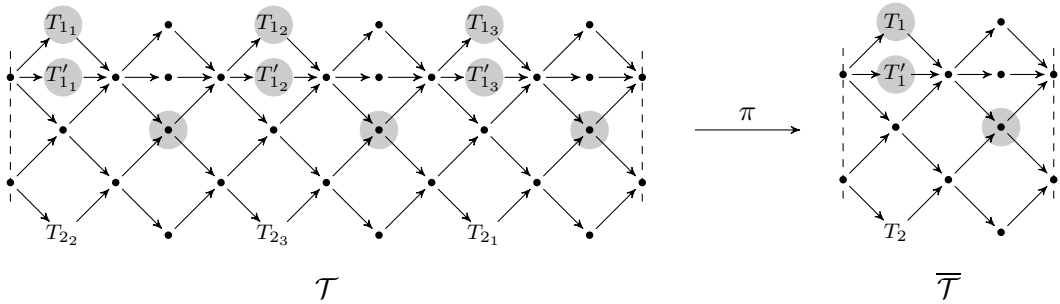
{\begin{tikzpicture}[scale=0.7,yscale=-1]  
\node at (6,5) {$\cT$};
 \foreach \x in {0,...,6}
  \foreach \y in {2,4}
   \node (\x-\y) at (\x*2,\y-1) [vertex] {};
 \foreach \x in {0,...,5}
  \foreach \y in {1,3,5}
   \node (\x-\y) at (\x*2+1,\y-1)  [vertex] {};
 \foreach \x in {0,...,5}
  \foreach \y in {6}
   \node (\x-\y) at (\x*2+1,1)  [vertex] {};   
 \replacevertex[fill1]{(0-1)} {[tvertex] {$T_{1_1}$}}
 \replacevertex[fill1]{(0-6)} {[tvertex] {$T'_{1_1}$}}
 \replacevertex[fill1]{(2-1)} {[tvertex] {$T_{1_2}$}}
 \replacevertex[fill1]{(2-6)} {[tvertex] {$T'_{1_2}$}}
\replacevertex[fill1]{(4-1)} {[tvertex] {$T_{1_3}$}}
 \replacevertex[fill1]{(4-6)} {[tvertex] {$T'_{1_3}$}}
 \replacevertex[fill1]{(1-3)} {[vertex] {}}
 \replacevertex[fill1]{(3-3)} {[vertex] {}}
 \replacevertex[fill1]{(5-3)} {[vertex] {}}
 \replacevertex{(0-5)} {[tvertex] {$T_{2_2}$}}
 \replacevertex{(2-5)} {[tvertex] {$T_{2_3}$}}
 \replacevertex{(4-5)} {[tvertex] {$T_{2_1}$}}
 \foreach \xa/\xb in {0/1,1/2,2/3,3/4,4/5,5/6}
  \foreach \ya/\yb in {2/1,2/6,2/3,4/3,4/5}
   {
    \draw [->] (\xa-\ya) -- (\xa-\yb);
    \draw [->] (\xa-\yb) -- (\xb-\ya);
   }
\draw [dashed] (0,0.5) -- (0,3.5); 
\draw [dashed] (12,0.5) -- (12,3.5); 
\draw [->] (13,2) -- node [above] {$\pi$} (15,2);
 \end{tikzpicture}
\quad 
 \begin{tikzpicture}[scale=0.7,yscale=-1]  
\node at (2,5) {$\ocT$};
 \foreach \x in {0,1,2}
  \foreach \y in {2,4}
   \node (\x-\y) at (\x*2,\y-1) [vertex] {};
 \foreach \x in {0,1}
  \foreach \y in {1,3,5}
   \node (\x-\y) at (\x*2+1,\y-1)  [vertex] {};
 \foreach \x in {0,1}
  \foreach \y in {6}
   \node (\x-\y) at (\x*2+1,1)  [vertex] {};   
 \replacevertex[fill1]{(0-1)} {[tvertex] {$T_1$}}
 \replacevertex[fill1]{(0-6)} {[tvertex] {$T'_1$}}
 \replacevertex[fill1]{(1-3)} {[vertex] {}}
 \replacevertex{(0-5)} {[tvertex] {$T_2$}}
 \foreach \xa/\xb in {0/1,1/2}
  \foreach \ya/\yb in {2/1,2/6,2/3,4/3,4/5}
   {
    \draw [->] (\xa-\ya) -- (\xa-\yb);
    \draw [->] (\xa-\yb) -- (\xb-\ya);
   }
\draw [dashed] (0,0.5) -- (0,3.5); 
\draw [dashed] (4,0.5) -- (4,3.5); 
 \end{tikzpicture}}
\]

 In the figure above, the subfactor categories $\cU$ and $\ocU$ are presented in light grey. They correspond to the $2$-CY triangulated categories $\cA_{3,1}$ and $\cA_{3,3}$, respectively (see Definition~\ref{def.2cy.cats}). At the level of quivers, we obtain the following picture:

\[\scalebox{1}{ \begin{tikzpicture}[scale=0.8]
\node at (-3,0) {$Q:$};
\node (1-1) at (2,0) {$1_1$};
\node (2-1) at (1,1.73) {$2_1$};
\node (1-2) at (-1,1.73) {$1_2$};
\node (2-2) at (-2,0) {$2_2$};
\node (1-3) at (-1,-1.73) {$1_3$};
\node (2-3) at (1,-1.73) {$2_3$};
\draw[->] (1-1) -- node [auto] {$\alpha_1$} (1-2);
\draw[->] (1-2) -- node [auto] {$\alpha_2$} (1-3);
\draw[->] (1-3) -- node [auto] {$\alpha_3$} (1-1);
\draw[->] (1-2) -- node [auto] {$\beta_1$} (2-1);
\draw[->] (1-3) -- node [auto] {$\beta_2$} (2-2);
\draw[->] (1-1) -- node [auto] {$\beta_3$} (2-3);
\draw[->] (2-1) -- node [auto] {$\gamma_1$} (1-1);
\draw[->] (2-2) -- node [auto] {$\gamma_2$} (1-2);
\draw[->] (2-3) -- node [auto] {$\gamma_3$} (1-3);
\node at (9,0) {$:\oQ$};
\node (1) at (6,0) {$1$};
\node (2) at (8,0) {$2$};
 \draw [->] (1) .. controls (5,0.75) and (5,-0.75) .. node [left] {$\alpha$} (1);
\draw [->] (6.25,-0.1) -- node [below] {$\beta$} (7.75,-0.1);
\draw [->] (7.75,0.1) -- node [above] {$\gamma$} (6.25,0.1);
\draw [->] (0,-2.73) -- node [right] {$\mu_{\left\{1_1,1_2,1_3\right\}}$} (0,-4.73);
\draw [->] (3,0) -- node [above] {$\pi$} (4.5,0);
\draw [->] (7,-1) -- node [right] {$\mu_{1}$} (7,-6.5);
 \pgftransformshift{\pgfpoint{0cm}{-7.5cm}};
\node at (-3,0) {$Q':$};
\node (1-1) at (2,0) {$1'_1$};
\node (2-1) at (1,1.73) {$2_1$};
\node (1-2) at (-1,1.73) {$1'_2$};
\node (2-2) at (-2,0) {$2_2$};
\node (1-3) at (-1,-1.73) {$1'_3$};
\node (2-3) at (1,-1.73) {$2_3$};
\draw[->] (1-1) -- node [auto] {$\alpha_1$} (1-2);
\draw[->] (1-2) -- node [auto] {$\alpha_2$} (1-3);
\draw[->] (1-3) -- node [auto] {$\alpha_3$} (1-1);
\draw[->] (1-2) -- node [auto] {$\beta_1$} (2-1);
\draw[->] (1-3) -- node [auto] {$\beta_2$} (2-2);
\draw[->] (1-1) -- node [auto] {$\beta_3$} (2-3);
\draw[->] (2-1) -- node [auto] {$\gamma_1$} (1-1);
\draw[->] (2-2) -- node [auto] {$\gamma_2$} (1-2);
\draw[->] (2-3) -- node [auto] {$\gamma_3$} (1-3);
\node at (10,0) {$:\oQ'=\pi(Q')$};
\node (1) at (6,0) {$1'$};
\node (2) at (8,0) {$2$};
 \draw [->] (1) .. controls (5,0.75) and (5,-0.75) .. node [left] {$\alpha$} (1);
\draw [->] (6.25,-0.1) -- node [below] {$\beta$} (7.75,-0.1);
\draw [->] (7.75,0.1) -- node [above] {$\gamma$} (6.25,0.1);
\draw [->] (3,0) -- node [above] {$\pi$} (4.5,0);
\end{tikzpicture}
}\]
 Here, $Q$ and $\oQ$ denote the quivers of the endomorphism rings of $\uT$ and $\dT$ respectively. Applying the mutation rule for cycles to $Q$ we obtain the quiver $Q'$. Then mutating $\oQ$ at vertex $1$ we obtain $\oQ'$, which is isomorphic to the quiver given by $\pi(Q')$. The potential of $\oQ$ (resp.\ $\oQ'$) is determined by the potential of $Q$ (resp.\ $Q'$).
\end{example}

\begin{example}\label{ex.mut.loop.a9}
Let $\cT=\cA_{9,1}$ be the cluster category of $\bbA_9$ and $\ocT=\cA_{9,3}$ the covering of order $3$ (see the diagram below -- also see Definition~\ref{def.2cy.cats} for notation). Then we have a Galois covering $\pi\from \cT \to \ocT$ of $2$-CY triangulated categories. Let $\dT=T_1\oplus T_2 \oplus T_3$ be the cluster-tilting object in $\ocT$ shown below. We denote by $\dT'=T'_1\oplus T_2 \oplus T_3$ the resulting cluster-tilting object obtained by replacing $T_1$ in $\dT$. 
 \[
 \scalebox{1}{\begin{tikzpicture}[scale=0.6,yscale=-1]  
  \fill [fill1] (0,4.6) arc (-90:90:12pt);
  \fill [fill1] (4,5.4) arc (90:270:12pt);
 \foreach \x in {0,1,2}
  \foreach \y in {1,3,5,7,9}
   \node (\x-\y) at (\x*2,\y) [vertex] {};
 \foreach \x in {0,1}
  \foreach \y in {2,4,6,8}
   \node (\x-\y) at (\x*2+1,\y)  [vertex] {};
 \replacevertex{(0-1)} {[tvertex] {$T_{3}$}}
 \replacevertex{(2-1)} {[tvertex] {$T_{2}$}}
 \replacevertex[fill1]{(1-3)} {[tvertex] {$T_{1}$}}
  \replacevertex{(2-9)} {[tvertex] {$T_{3}$}}
  \replacevertex{(0-9)} {[tvertex] {$T_{2}$}}
 \replacevertex[fill1]{(1-7)} {[tvertex] {$T'_{1}$}}
 \foreach \xa/\xb in {0/1,1/2}
  \foreach \ya/\yb in {1/2,3/2,3/4,5/4,5/6,7/6,7/8,9/8}
   {
    \draw [->] (\xa-\ya) -- (\xa-\yb);
    \draw [->] (\xa-\yb) -- (\xb-\ya);
   }
\draw [dashed] (0,.5) -- (0,0.75); 
\draw [dashed] (0,1.25) -- (0,8.75); 
\draw [dashed] (0,9.25) -- (0,9.5); 
\draw [dashed] (4,.5) -- (4,0.75);
\draw [dashed] (4,1.25) -- (4,8.75);
\draw [dashed] (4,9.25) -- (4,9.5);
 \end{tikzpicture}}
\]
The lift $\uT$ (resp.\ $\uT'$) of $\dT$ (resp.\ $\dT'$) to $\cT$ is shown in Example~\ref{example.a9}. The mutation of $Q$, the quiver of $\uT$, at the cycle $\{1_1,1_2,1_3\}$ in order to obtain $Q'$, the quiver of $\uT'$, is illustrated in Example~\ref{example.quivers.a9}. Then after mutating $\oQ$ at vertex $1$ we obtain $\oQ'$, which is isomorphic to the quiver given by $\pi(Q')$. Therefore  we have the following commutative diagram.
\[
\scalebox{1}{ \begin{tikzpicture}[scale=0.8,yscale=-1]
 \node at (4,7.464) {$\oQ$};
 \node at (4,3.7) {$\alpha$}; 
\node (M3) at (4,4.732)  {$1$};
 \node (F13) at (3,6.464)  {$2$};
 \node (F23) at (5,6.464)  {$3$};
\draw [->] (M3) .. controls (3,3.732) and (5,3.732) .. (M3);
 \draw [->] (M3) -- node [auto,swap] {$\beta$} (F13);
 \draw [->] (F13) -- node [auto,swap] {$\gamma$} (F23);
 \draw [->] (F23) -- node [auto,swap] {$\delta$} (M3);
\draw [->] (6,6.5) to [bend left] node [below] {$\mu_1$} (12,6.5);
\draw [->] (6,5.5) to [bend right] node [above] {$\pi\comp\mu_{\left\{1_1,1_2,1_3 \right\}}\comp\pi^{-1}$} (12,5.5);
  \pgftransformshift{\pgfpoint{10cm}{0cm}};
 \node at (4,7.464) {$\oQ'=\pi(Q')$};
 \node at (4,3.7) {$\alpha$}; 
 \node (M3) at (4,4.732)  {$1$};
 \node (F13) at (5,6.464)  {$3$};
 \node (F23) at (3,6.464)  {$2$};
\draw [->] (M3) .. controls (3,3.732) and (5,3.732) .. (M3);
 \draw [->] (M3) -- node [auto] {$[\delta\alpha^2]^t$} (F13);
 \draw [->] (F13) -- node [auto] {$[\delta\alpha\beta]$} (F23);
 \draw [->] (F23) -- node [auto] {$[\alpha^2\beta]^t$} (M3);
 \end{tikzpicture}}
\]
The potential of $\oQ$ (resp.\ $\oQ'$) is determined by the potential of $Q$ (resp.\ $Q'$). 
\end{example}

\section{2 CY-tilted algebras of finite type}\label{section.2CY.finite.type}
 
  In this section, we give a description of the mutation classes of the 2-CY tilted algebras (which are not cluster-tilted) coming from standard algebraic 2-CY triangulated categories with a finite number of indecomposables. (A category is called standard if it is equivalent to the mesh category of its Auslander-Reiten quiver.) We will see that these types of algebras always satisfy our setup for mutating at loops. Thus by using our mutation rule developed in the previous section, we will be able to mutate at any vertex.

 Let $k$ be an algebraically closed field. The 2-CY tilted algebras of finite type appear as endomorphism rings of cluster-tilting objects in $k$-linear 2-CY triangulated categories $\cT$ with a finite number of indecomposables. In \cite{BIKR} the authors prove that the existence of cluster-tilting objects in these categories follows from the shape of their AR-quiver. These shapes were described in \cite{A1,XZ}. 

First, let us fix a numbering and an orientation of the simply-laced Dynkin quivers.

\[\scalebox{1}{ \begin{tikzpicture}[xscale=1.5,yscale=-1.2]
\node (A) at (0,0) {$\bbA_n:$};
\node (A1) at (1,0) {$1$};
\node (A2) at (2,0) {$2$};
\node (A3) at (3,0) {$3$};
\node (A4) at (4,0) {$\cdots$};
\node (A5) at (5.2,0) {$n-1$};
\node (A6) at (6.4,0) {$n$};
\node (D) at (0,1) {$\bbD_n:$};
\node (D1) at (1,1) {$1$};
\node (D2) at (2,1) {$2$};
\node (D3) at (3,1) {$3$};
\node (D4) at (4,1) {$\cdots$};
\node (D5) at (5.2,1) {$n-2$};
\node (D6) at (6.4,0.5) {$n-1$};
\node (D7) at (6.4,1.5) {$n$};
\node (E) at (0,2.5) {$\bbE_n:$};
\node (E1) at (1,2.5) {$1$};
\node (E2) at (2,2.5) {$2$};
\node (E3) at (3,2.5) {$3$};
\node (E4) at (3,1.5) {$4$};
\node (E5) at (4,2.5) {$5$};
\node (E6) at (5.2,2.5) {$\cdots$};
\node (E7) at (6.4,2.5) {$n$};
\draw[->] (A1) -- (A2);
\draw[->] (A2) -- (A3);
\draw[->] (A3) -- (A4);
\draw[->] (A4) -- (A5);
\draw[->] (A5) -- (A6);
\draw[->] (D1) -- (D2);
\draw[->] (D2) -- (D3);
\draw[->] (D3) -- (D4);
\draw[->] (D4) -- (D5);
\draw[->] (D6) -- (D5);
\draw[->] (D7) -- (D5);
\draw[->] (E2) -- (E1);
\draw[->] (E3) -- (E2);
\draw[->] (E3) -- (E4);
\draw[->] (E3) -- (E5);
\draw[->] (E5) -- (E6);
\draw[->] (E6) -- (E7);
\end{tikzpicture}} \]

For the definition of a translation quiver, we refer the reader to \cite[Chapter~VII]{ARS}.

\begin{definition} \label{automorphism.def}  
For a Dynkin quiver $\Delta$ we define the following automorphisms of the translation quiver $(\bbZ \Delta,\tau)$ (in all cases $S$ is the combinatorial description of the shift functor):
\begin{enumerate}
\item If $\Delta=\bbA_n$, define $S(i,p)= (i+p,n+1-p)$ where $i\in\bbZ$ and $p$ is a vertex of $\bbA_n$. Moreover, define
 \[ \phi = \begin{cases}
              \tau^{\frac{n}{2}}S & \text{if $n$ is even,} \\ 
              \tau^{\frac{n+1}{2}}S & \text{if $n$ is odd.}
            \end{cases}\]
 Observe that for $n$ even $\phi^2=\tau^{-1}$, and for $n$ odd $\phi^2=1$.
\item If $\Delta=\bbD_n$, define $\phi$ to be the automorphism exchanging vertices $n$ and $n-1$, and let 
 \[ S = \begin{cases}
              \tau^{-n+1} & \text{if $n$ is even,} \\ 
              \tau^{-n+1}\phi & \text{if $n$ is odd.}
            \end{cases}\] 
\item If $\Delta=\bbE_8$ then $S=\tau^{-15}$.
\end{enumerate}
\end{definition}
The following theorem has been adapted to our setup.

\begin{theorem}[{\cite[Theorem~8.2~(1)]{BIKR}}] \label{thm.classification} Let $\cT$ be a 2-CY triangulated category, not a cluster category, with a finite number of indecomposables. Then $\cT$ has a cluster-tilting object if and only if the AR-quiver of $\cT$ is $\bbZ\Delta/g$ for a Dynkin diagram $\Delta$ and $g\in \Aut\bbZ\Delta$ in the table below.  
\begin{center}
  \newcommand{\rb}[1]{\raisebox{2ex}[0pt]{#1}}
  \renewcommand{\arraystretch}{1.4} 
  \setlength{\tabcolsep}{3mm} 
   \begin{tabular}{|c|c|c|c|c|}\hline 
{\boldmath $\Delta$} & \boldmath $\Aut \bbZ \Delta$ & \boldmath $g$ & \multicolumn{2}{|c|}{\bf Restrictions}  \\ \hline 
 & $\bbZ $ & $\phi^{\frac{n+3}{3}}$  & $n$ even &    \\ \cline{2-4}
    \rb{$\bbA_{n}$} & $\bbZ \times \bbZ/2\bbZ$ & $\tau^{\frac{n+3}{6}} \phi $ & $n$ odd &  \rb{$3|n$}  \\ \hline
  & $\bbZ \times \bbZ/2\bbZ$ & $\tau^m \phi^{\om} $ & $n$ even &  \\ \cline{2-4}
\rb{$\bbD_{n}$} & $\bbZ \times \bbZ/2\bbZ$ & $\tau^m \phi $ & $n$ odd & \rb{$m|n$, $n>4$} \\ \hline
{$\bbD_{4}$} &  $\bbZ \times S_3$ & $\tau^m\sigma $ &  \multicolumn{2}{|c|}{$  m|4 , \,\sigma^{\frac{4}{m}}=1, \, (m,\sigma)\ne (1,1)$ } \\ \hline
 {$\bbE_{8}$} & $\bbZ$ & $\tau^8$ &   \multicolumn{2}{|c|}{} \\ \hline
   \end{tabular}
\end{center}
Here $S_3$ is the permutation group of three elements and $\phi$ is the automorphism of $\bbZ\Delta$ as in Definition~\ref{automorphism.def}.
\end{theorem}

 Let $F$ be the automorphism $\tau^{-1}S$ of $\bbZ\Delta$ and let $n$ be a positive integer as in Theorem~\ref{thm.classification}. Then we note that in the case $\bbA_{n}$ (resp.\ $\bbD_{n}$, $\bbE_8$) we have $g^3=F$ (resp.\ $g^{\frac{n}{m}}=F$, $g^2=F$). Then we define the following.
  
\begin{definition} \label{def.2cy.cats}
Let $\ell$ and $n$ be two positive integers such that $\ell$ divides $n$. We denote by $\cA_{n,\ell}$ (resp.\ $\cD_{n,\ell}$ and $\cE_{n,\ell}$) the standard algebraic $2$-CY triangulated category having AR-quiver $\bbZ \bbA_n /g$ (resp.\ $\bbZ \bbD_n /g$ and  $\bbZ \bbE_n /g$), where in each case we have that $g^{\ell} = F$.
\end{definition}

\begin{remark}
This definition only makes sense if we ask our triangulated category to be standard and algebraic. By \cite[7.0.5]{A1}, it is known that these categories are unique up to a triangle isomorphism. With the above definitions, the cluster categories of type $\bbA_n$, $\bbD_n$ and $\bbE_8$ are denoted  by $\cA_{n,1}$, $\cD_{n,1}$ and $\cE_{8,1}$ respectively.
\end{remark}

We can now give a simpler reformulation of Theorem~\ref{thm.classification} for the standard algebraic case. 

\begin{theorem} \label{thm.classification_new}
Let $\cT$ be a standard algebraic  2-CY triangulated category, not a cluster category, with a finite number of indecomposables. Then $\cT$ has a cluster-tilting object if and only if $\cT$ is either the category $\cA_{3n,3}$ for some $n \geq 1$, the category $\cD_{n \ell, \ell}$ for $n$ and $\ell$ such that $n \ell \geq 4$, or the category $\cE_{8,2}$.
\end{theorem}

\begin{proof} This is a direct consequence of \cite[7.0.5]{A1} and Theorem~\ref{thm.classification}. 
\end{proof}

 The following definition is useful for the description of the $2$-CY tilted algebras of finite type. It was first introduced in \cite{Dagfinn} in order to describe the cluster-tilted algebras of type $\bbD$.
\begin{definition} \label{def.connecting.vertex}
Let $Q$ be the quiver of a cluster-tilted algebra of type $\bbA$. A vertex of $Q$ is called a \emph{connecting vertex} if
\begin{enumerate}
 \item there are at most two arrows adjacent to it, and
 \item whenever there are two arrows adjacent to it, the vertex is on a $3$-cycle.
\end{enumerate} 
\end{definition}

We are now ready to present the main theorem of this section.

\begin{theorem}
Let $T$ be a cluster-tilting object in a standard algebraic  2-CY triangulated category $\cT$ of finite type, not a cluster category. Then $\End_{\cT} (T)$ is depicted in Figure~\ref{figure.2CYfinite}.
\end{theorem}  
  
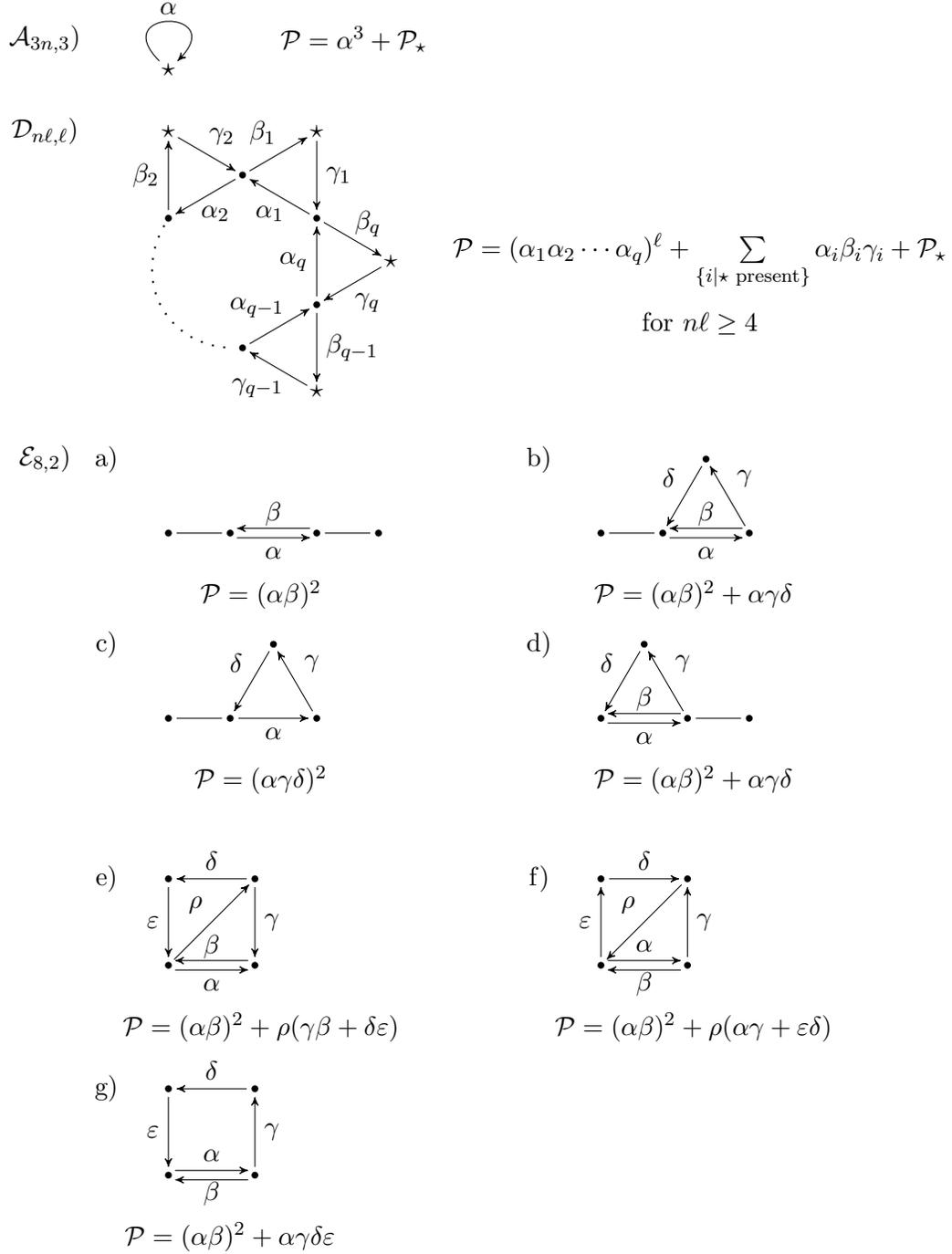
\begin{figure}
\[\scalebox{1}{ \begin{tikzpicture}[scale=0.9,yscale=-1]
 \node at (-1,.5) {$\cA_{3n,3}$)};
 \node (V) at (1,1) [inner sep=1pt] {$\star$};
 \draw [->] (V) .. controls (0,0) and (2,0) .. (V);
 \node at (1,0) {$\alpha$};
 \node at (4,0.5) {$\cP= \alpha^3 + \cP_\star$};
 \pgftransformshift{\pgfpoint{0cm}{-6cm}};
 \node at (-1,8) {$\cD_{n\ell, \ell}$)};
 \node (V3) at (1,9.4) [vertex] {};
 \node (V2) at (2.2,8.7) [vertex] {};
 \node (V1) at (3.4,9.4) [vertex] {};
 \node (V5) at (3.4,10.8) [vertex] {};
 \node (V4) at (2.2,11.5) [vertex] {};
 \node (V6) at (1,8) [inner sep=1pt] {$\star$};
 \node (V7) at (3.4,8) [inner sep=1pt] {$\star$};
 \node (V8) at (4.6,10.1) [inner sep=1pt] {$\star$};
 \node (V9) at (3.4,12.2) [inner sep=1pt] {$\star$};
 \draw [->] (V1) -- node [auto,xshift=2mm] {$\alpha_1$}(V2);
 \draw [->] (V2) -- node [auto,xshift=-2mm] {$\alpha_2$} (V3);
 \draw [thick, loosely dotted] (V3) .. controls (0.4,10.45) and (1,11.5) .. (V4);
 \draw [->] (V4) -- node [auto,xshift=2mm] {$\alpha_{q-1}$}(V5);
 \draw [->] (V5) -- node [auto] {$\alpha_q$}(V1);
 \draw [->] (V3) to node [auto] {$\beta_2$} (V6);
 \draw [->] (V6) to node [auto,xshift=-1mm] {$\gamma_2$}(V2);
 \draw [->] (V2) to node [auto,xshift=1mm] {$\beta_1$}(V7);
 \draw [->] (V7) to node [auto] {$\gamma_1$}(V1);
 \draw [->] (V1) to node [auto,xshift=-1mm,yshift=-1mm] {$\beta_q$}(V8);
 \draw [->] (V8) to node [auto,xshift=-1mm] {$\gamma_q$}(V5);
 \draw [->] (V5) to node [auto] {$\beta_{q-1}$}(V9);
 \draw [->] (V9) to node [auto,xshift=2mm] {$\gamma_{q-1}$}(V4);
 \node at (9.6,10.1) {$\cP= (\alpha_1 \alpha_2 \cdots
   \alpha_q)^{\ell} + \Sum_{ \{ i \mid \star \text{ present} \} } \alpha_i\beta_i\gamma_i + \cP_\star$};
 \node at (9.6,11.1) {for $n\ell \ge 4$};
 \pgftransformshift{\pgfpoint{0cm}{0.5cm}};
 \node at (-1,12.8) {$\cE_{8,2}$)};
 \node at (0,12.8) {a)};
 \node (W1) at (1,14) [vertex] {};
 \node (W2) at (2,14) [vertex] {};
 \node (W3) at (3.4,14) [vertex] {};
 \node (W4) at (4.4,14) [vertex] {};
 \node at (2.5,15) {$\cP = (\alpha\beta)^2$};
 \draw (W1) -- (W2);
 \draw [->] (W2.325) to node[auto,swap] {$\alpha$} (W3.215);
 \draw [->] (W3.145) to node[auto,swap,yshift=-1mm] {$\beta$} (W2.35);
 \draw (W3) -- (W4);
 \node at (7,12.8) {b)};
 \node (Z1) at (8,14) [vertex] {};
 \node (Z2) at (9,14) [vertex] {};
 \node (Z3) at (10.4,14) [vertex] {};
 \node (Z4) at (9.7,12.8) [vertex] {};
 \node at (9.5,15) {$\cP = (\alpha\beta)^2 + \alpha\gamma\delta$};
 \draw  (Z1) -- (Z2); 
 \draw [->] (Z2.325) to node[auto,swap] {$\alpha$} (Z3.215);
 \draw [->] (Z3.145) to  node[auto,swap,yshift=-1mm] {$\beta$} (Z2.35);
 \draw [->] (Z3) to node[auto,swap] {$\gamma$} (Z4); 
 \draw [->] (Z4) to node[auto,swap] {$\delta$} (Z2); 
 \node at (0,15.8) {c)};
 \node (Z1) at (1,17) [vertex] {};
 \node (Z2) at (2,17) [vertex] {};
 \node (Z3) at (3.4,17) [vertex] {};
 \node (Z4) at (2.7,15.8) [vertex] {};
 \node at (2.5,18) {$\cP = (\alpha\gamma\delta)^2$};
 \draw  (Z1) -- (Z2); 
 \draw [->] (Z2) to  node[auto,swap] {$\alpha$} (Z3);
 \draw [->] (Z3) to  node[auto,swap] {$\gamma$} (Z4); 
 \draw [->] (Z4) to  node[auto,swap] {$\delta$} (Z2); 
 \node at (7,15.8) {d)};
 \node (ZZ1) at (10.4,17) [vertex] {};
 \node (ZZ2) at (8,17) [vertex] {};
 \node (ZZ3) at (9.4,17) [vertex] {};
 \node (ZZ4) at (8.7,15.8) [vertex] {};
 \node at (9.5,18) {$\cP = (\alpha\beta)^2 + \alpha\gamma\delta$};
 \draw  (ZZ1) -- (ZZ3); 
 \draw [->] (ZZ2.325) to  node[auto,swap] {$\alpha$} (ZZ3.215);
 \draw [->] (ZZ3.145) to  node[auto,swap,yshift=-1mm] {$\beta$} (ZZ2.35);
 \draw [->] (ZZ3) to  node[auto,swap] {$\gamma$} (ZZ4); 
 \draw [->] (ZZ4) to  node[auto,swap] {$\delta$} (ZZ2); 
 \node at (0,19.6) {e)};
 \node (X1) at (1,21) [vertex] {};
 \node (X2) at (2.4,21) [vertex] {};
 \node (X3) at (2.4,19.6) [vertex] {};
 \node (X4) at (1,19.6) [vertex] {};
 \node at (2.5,22) {$\cP = (\alpha\beta)^2 + \rho(\gamma\beta + \delta\varepsilon)$};
 \draw [->] (X1.325) to node[auto,swap] {$\alpha$} (X2.215);
 \draw [->] (X2.145) to  node[auto,swap,yshift=-1mm] {$\beta$} (X1.35);
 \draw [->] (X3) --  node[auto] {$\gamma$} (X2);
 \draw [->] (X3) --  node[auto,swap] {$\delta$} (X4);
 \draw [->] (X1) --  node[auto] {$\rho$} (X3);
 \draw [->] (X4) --  node[auto,swap] {$\varepsilon$} (X1);
 \node at (7,19.6) {f)};
 \node (X1) at (8,21) [vertex] {};
 \node (X2) at (9.4,21) [vertex] {};
 \node (X3) at (9.4,19.6) [vertex] {};
 \node (X4) at (8,19.6) [vertex] {};
 \node at (9.5,22) {$\cP = (\alpha\beta)^2 + \rho(\alpha\gamma + \varepsilon\delta)$};
 \draw [->] (X1.35) to node[auto] {$\alpha$} (X2.145);
 \draw [->] (X2.215) to  node[auto,yshift=1mm] {$\beta$} (X1.325);
 \draw [->] (X2) --  node[auto,swap] {$\gamma$} (X3);
 \draw [->] (X4) --  node[auto] {$\delta$} (X3);
 \draw [->] (X3) --  node[auto,swap] {$\rho$} (X1);
 \draw [->] (X1) --  node[auto] {$\varepsilon$} (X4);
 \node at (0,23) {g)};
 \node (Y1) at (1,24.4) [vertex] {};
 \node (Y2) at (2.4,24.4) [vertex] {};
 \node (Y3) at (2.4,23) [vertex] {};
 \node (Y4) at (1,23) [vertex] {};
 \node at (2,25.4) {$\cP = (\alpha\beta)^2 + \alpha\gamma\delta\varepsilon$};
 \draw [->] (Y1.35) to node[auto] {$\alpha$} (Y2.145);
 \draw [->] (Y2.215) to  node[auto,yshift=1mm] {$\beta$} (Y1.325);
 \draw [->] (Y2) -- node[auto,swap] {$\gamma$} (Y3); 
 \draw [->] (Y3) -- node[auto,swap] {$\delta$} (Y4); 
 \draw [->] (Y4) -- node[auto,swap] {$\varepsilon$} (Y1);
\end{tikzpicture}} \]
\caption{\label{figure.2CYfinite} These are the 2-CY tilted algebras of finite type that are not cluster-tilted, organized by their mutation class. The relations are given in each case by the potential $\cP$. For the cases $\cA_{3n,3}$ and $\cD_{n\ell,\ell}$, the vertex $\star$ is a connecting vertex (see Definition~\ref{def.connecting.vertex}) where a cluster-tilted algebra of type $\bbA$ is glued, and the term $\cP_\star$ corresponds to the sum of the potentials of all  cluster-tilted algebras of type $\bbA$ attached at $\star$. In case $\cD_{n\ell,\ell}$, the  vertices $\star$ may or may not be present, and thus, the  corresponding $\beta$ and $\gamma$ arrows disappear (However, in the case $q=1$, we must have one connecting vertex $\star$).}
\end{figure}

\begin{proof}
Observe that we have a covering functor $\pi \from \cC_\Delta\to \cT$, where $\cC_\Delta$ is a cluster category of Dynkin type $\Delta$. We proceed by finding the cluster-tilted algebras in $\cC_\Delta$ which are a cover of the 2-CY tilted algebras in $\cT$. By Theorem~\ref{thm.classification_new} we have three cases:
 
{\bf Case $\cT=\cA_{3n,3}$.} Using the geometric description of the cluster category of type $\bbA$, we know that cluster-tilting objects correspond to triangulations of a regular $(3n+3)$-gon (see \cite{CCS1,I}). Observe that the automorphism $g$ corresponds to a rotation by $2 \pi / 3$. We want to find all the triangulations of the polygon invariant under $g$. 

Assume we are given such a $g$-invariant triangulation. Let $d$ be the longest diagonal which is part of the triangulation. If it covers an angle of more than $2 \pi / 3$ then the diagonals $d$ and $g \cdot d$ intersect in their interior, a contradiction. If all diagonals cover an angle of less than $2 \pi / 3$ then the shape which contains the center of the polygon cannot be a triangle, also a contradiction. Thus $d$ covers an angle of exactly $2 \pi / 3$, and $d$, $g \cdot d$, and $g^2 \cdot d$ form a regular triangle in the center of the regular $(3n+3)$-gon.

Next, we note that the remaining diagonals correspond to three identical triangulations of an $(n+2)$-gon, that is, a cluster-tilting object in the cluster category of type $\bbA_{n-1}$.

Now projecting back to $\cT$, we obtain a $2$-CY tilted algebra with a loop $\alpha$ corresponding to the orbit of the diagonal $d$, attached to a cluster-tilted algebra of type $\bbA$. Since the orbit of the diagonal $d$ is a triangle, the loop satisfies the relation $\alpha^2=0$, thus obtaining the quiver with relations depicted in Figure~\ref{figure.2CYfinite}~$\cA_{3n,n})$.

{\bf Case $\cT=\cD_{n\ell,\ell}$.} We can assume  $n\ell\ge 4$. In this case, cluster-tilting objects in the cluster category of $\bbD_{n\ell}$ correspond to (tagged) triangulations of a punctured $n\ell$-gon (see \cite{S}). Here the automorphism $g$ corresponds to a rotation by $2\pi/ \ell$  composed with $\phi^{n}$ where $\phi$ is the automorphism that exchanges the tagged diagonals with non-tagged diagonals. 

Assume we are given a $g$-invariant triangulation of the punctured $n \ell$-gon. By definition there is at least one diagonal connecting the puncture to the polygon. It follows that there are at least $\ell$ vertices of the polygon connected to the puncture by diagonals (the $\ell$ $g$-translates of any given one).

It follows that no vertex is connected two the puncture by more than one diagonal (hence we may ignore the question if edges are tagged or not).

Now consider all diagonals connecting vertices of the polygon to the puncture. Clearly they form a cycle of length $q \ell$ in the quiver of the cluster tilted algebra of type $\mathbb{D}_{n\ell}$, and hence a cycle of length $q$ in the quiver of our 2-CY tilted algebra of type $\mathcal{D}_{n\ell, \ell}$.

If two consecutive such diagonals start in consecutive vertices of the polygon, then the corresponding arrow of the $q$-cycle is not involved in any further cycles of the quiver of the 2-CY tilted algebra.

If two consecutive diagonals ending in the puncture start in vertices of the polygon which are further appart, then these vertices are connected by another diagonal. Moreover there is some triangulation of the part of the polygon cut off by this other diagonal. In the quiver of the 2-CY tilted algebra this means that the arrow of the $q$-cycle is involved in one further triangle, which connects it to the connecting vertex of some quiver of a cluster tilted algebra of type $\mathbb{A}$.

{\bf Case $\cT=\cE_{8,2}$.} This is a finite combinatorial task. It is simplified by the following observations:
\begin{enumerate} 
\item Numbering the $\tau$-orbits starting from the top most orbit and below, we have $8$ orbits, say $\sigma_1,\ldots,\sigma_8$ (see Figure~\ref{E8.numbering}).
\item The orbits $\sigma_1,\sigma_2$ and $\sigma_8$ are the only ones having exceptional objects. To a cluster-tilting object we can associate a triple $(a_1,a_2,a_8)$ of non-negative integers, where $a_i$ denotes the number of indecomposable summands in the orbit $\sigma_i$.
\item $a_i \in \{0,1,2\}$ for $i \in \{1,2,8\}$.
\end{enumerate}
 
 We consider the numbering of the indecomposable objects of $\cE_{8,2}$ as in Figure~\ref{E8.numbering}. The possible cluster-tilting objects up to a $\tau$-shift are illustrated in Table~\ref{table.E8}. Their endomorphism rings are depicted in Figure~\ref{figure.2CYfinite}~$\cE_{8,2})$.

\begin{table}[hb]
\begin{center}
  \renewcommand{\arraystretch}{1.4} 
  \setlength{\tabcolsep}{3mm} 
   \begin{tabular}{|c|c|c|}\hline 
\boldmath{$T$} & \boldmath $\End_\cT(T)$ & {\bf Type}  \\ \hline 
$ (0 \oplus 24,-,3 \oplus 27)$ & $g)$ & (2,0,2) \\ \hline
$(24 \oplus 48,-,27 \oplus 51)$ & $g)$& (2,0,2) \\ \hline
$(0 \oplus 24,-,3 \oplus 43)$ & $f)$ & (2,0,2)\\ \hline
$(24 \oplus 48,-,3 \oplus 27)$ & $f)$ & (2,0,2) \\ \hline 
$(0 \oplus 24,-,27 \oplus 51)$ & $e)$ & (2,0,2) \\ \hline
$(24 \oplus 48,-,11\oplus51)$ & $e)$ & (2,0,2) \\ \hline
$(24,28,11\oplus 51)$  & $c)$& (1,1,2)  \\ \hline
$(32,28,11 \oplus 51)$  & $c)$& (1,1,2)  \\ \hline
$(0 \oplus 24,28,51)$ & $d)$& (2,1,1) \\ \hline
$(0 \oplus 32,28,51)$  & $d)$& (2,1,1) \\ \hline
$ (24 \oplus 56,28,11)$  & $b)$ & (2,1,1)  \\ \hline
$(32 \oplus 56,28,11)$  & $b)$& (2,1,1) \\ \hline
$(0\oplus 24,28\oplus 60,-)$  & $a)$ & (2,2,0)  \\ \hline
$(0\oplus 32,28\oplus 60,-)$  & $a)$ & (2,2,0)  \\ \hline
$(24\oplus 56,28\oplus 60,-)$  & $a)$& (2,2,0)  \\ \hline
$(32\oplus 56,28\oplus 60,-)$  & $a)$ & (2,2,0)  \\ \hline
   \end{tabular}
\caption{\label{table.E8} Possible cluster-tilting objects up to $\tau$-shift in $\cE_{8,2}$.}
\end{center}
\end{table}

Thus the assertion from the theorem follows. \qedhere
\end{proof}

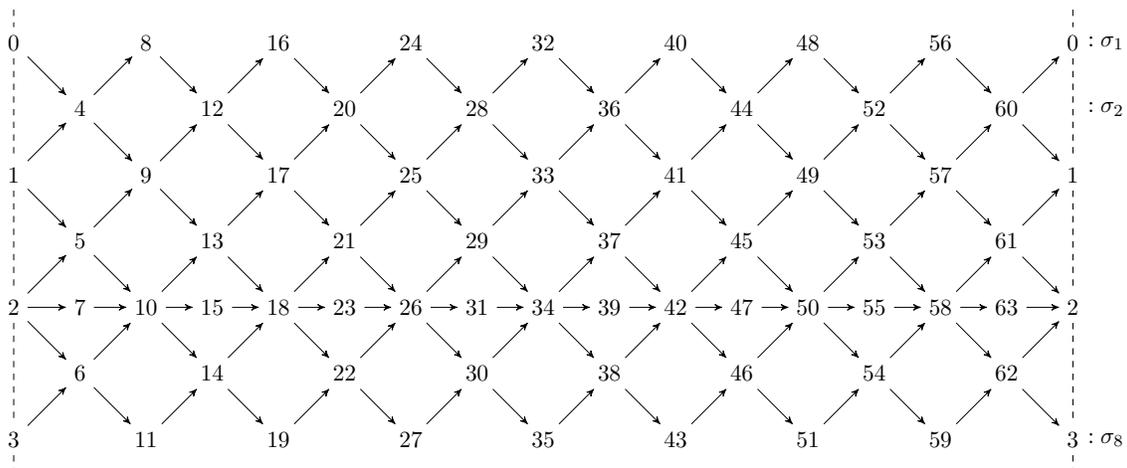
\begin{figure}[hb]
\[\scalebox{0.8}{ \begin{tikzpicture}[scale=1.1,yscale=-1]
 \foreach \x in {0,...,8}
  \foreach \y in {0,1,2,3}
   \node (\y-\x) at (\x*2,\y*2) {\pgfmathtruncatemacro{\NUMBER}{mod(\x*8+\y,64)}\NUMBER};
 \foreach \x in {0,...,7}
  \foreach \y in {4,5,6}
    \node (\y-\x) at (\x*2+1,\y*2-7) {\pgfmathtruncatemacro{\NUMBER}{mod(\x*8+\y,64)}\NUMBER};
 \foreach \x in {0,...,7}
  \foreach \y in {7}
   \node (\y-\x) at (\x*2+1,4) {\pgfmathtruncatemacro{\NUMBER}{mod(\x*8+\y,64)}\NUMBER};
\node at (16.5,0) {$:\sigma_1$};
\node at (16.5,1) {$:\sigma_2$};
\node at (16.5,6) {$:\sigma_8$};
 \foreach \xa/\xb in {0/1,1/2,2/3,3/4,4/5,5/6,6/7,7/8}
   \foreach \ya/\yb in {0/4,1/4,1/5,2/5,2/6,2/7,3/6}
   {
    \draw [->] (\ya-\xa) -- (\yb-\xa);
    \draw [->] (\yb-\xa) -- (\ya-\xb);
   }
\draw [dashed] (0,-0.5) -- (0-0.north);
\draw [dashed] (0-0.south) -- (1-0.north);
\draw [dashed] (1-0.south) -- (2-0.north);
\draw [dashed] (2-0.south) -- (3-0.north);
\draw [dashed] (3-0.south) -- (0,6.5); 
\draw [dashed] (16,-0.5) -- (0-8.north);
\draw [dashed] (0-8.south) -- (1-8.north);
\draw [dashed] (1-8.south) -- (2-8.north);
\draw [dashed] (2-8.south) -- (3-8.north);
\draw [dashed] (3-8.south) -- (16,6.5); 
\end{tikzpicture}} \]
\caption{\label{E8.numbering} Numbering of the vertices in the AR-quiver of the 2-CY triangulated category $\cE_{8,2}$.}
\end{figure}

\begin{remark} Using Theorem~\ref{theorem.mutation.loops}, we can now mutate at any vertex for a $2$-CY tilted algebra of finite type (see for instance Examples \ref{ex.mut.loop.d6} and \ref{ex.mut.loop.a9}). One can check that in these finite 2-CY categories, all cluster-tilting objects are mutation connected. We illustrate the mutation graph of $\cE_{8,2}$ in Figure~\ref{E8.mutation.component}.
\end{remark}

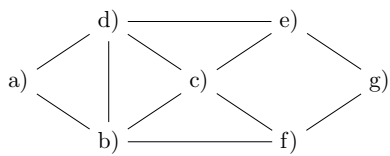
\begin{figure}
\[\scalebox{0.8}{ \begin{tikzpicture}[xscale=-1.5]
\node (a) at (2,1) {a)};
\node (b) at (1,0) {b)};
\node (c) at (0,1) {c)};
\node (d) at (1,2) {d)};
\node (e) at (-1,2) {e)};
\node (f) at (-1,0) {f)};
\node (g) at (-2,1) {g)};
\draw [-] (a) -- (b);
\draw [-] (a) -- (d);
\draw [-] (b) -- (c);
\draw [-] (c) -- (d);
\draw [-] (f) -- (g);
\draw [-] (g) -- (e);
\draw [-] (b) -- (f);
\draw [-] (d) -- (e);
\draw [-] (f) -- (c);
\draw [-] (c) -- (e);
\draw [-] (b) -- (d);
\end{tikzpicture}
}\] \caption{Mutation component of the 2-CY triangulated category $\cE_{8,2}$.} \label{E8.mutation.component}
\end{figure}

\clearpage

\section{Appendix} \label{appendix.ugly}
\subsection*{Proof of Theorem~\ref{rule.mut_min_cycle}} We have now the following cases:
\begin{itemize}
\item {$M'_{m,m}$.}  Using that $A^{l-1}A=A A^t = 1$, we see that $-A^{-1}(A - A^t)(-A) = A - A^t$. Thus, the arrows in $Q_{m,m}$ remain unchanged.
\item {$M'_{m,f}$.} Observing that $M'_{m,f} = -(M'_{f,m})^t$, it suffices to calculate just one of them.
\begin{align*}
 M'_{f, m} & = \left( (\Sum_{i=0}^{l-2}C_{i} \Sum^{i}_{j=0}A^j)(A-A^t) + C -B^t  \right)(-A) \\
& = \Sum_{i=0}^{l-2}C_{i}  \underbrace{ \left( -\Sum_{j=0}^i A^{j+2} + \Sum_{j=0}^i A^j - A \right) }_{=1-A^{i+1}-A^{i+2}} + \Sum_{i=0}^{l-3} \underbrace{B_{i}^t A}_{=C_{i}A^{i+2}} + B^t_{l-2} A \\
& = - \Sum_{i=0}^{l-2} C_{i} A^{i+1} + \Sum_{i=0}^{l-2} C_{i} - C_{l-2} \underbrace{A^l}_{=1} + B_{l-2}^t A  \\
& =  - \Sum_{i=0}^{l-3} \underbrace{C_{i} A^{i+1}}_{B_i^t} - C_{l-2}A^{l-1} + \Sum_{i=0}^{l-3} C_{i} + B_{l-2}^t A \\
& =  \Sum_{i=0}^{l-3} C_{i} - \Sum_{i=0}^{l-3} B_{i}^t + \left( A^{l-1}B_{l-2}\right)^t - C_{l-2}A^{l-1}. 
\end{align*}
 Thus any arrow $\gamma$ in $C_i$, with $1\le i \le l-3$ remains unchanged. Let $\gamma$ be an arrow in $C_{l-2}$. There is a path $\gamma \alpha^{l-1}$. We replace this arrow $\gamma$ by an arrow from the end of $\gamma \alpha^{l-1}$ to the start of $\gamma$. Apply the dual process for arrows in $Q_{m,f}$.

\item {$M'_{f,f}$.} This part of the matrix is composed of the following summands.
\begin{align*}
 M'_{f,f} =& \left(\Sum_{i=0}^{l-2}C_{i} \Sum^{i}_{j=0}A^j\right)(A-A^{-1})\left(\Sum_{i=0}^{l-2}C_{i} \Sum^{i}_{j=0}A^j\right)^t \\
& + \left(\Sum_{i=0}^{l-2}C_{i} \Sum^{i}_{j=0}A^j\right)(B - C^t) + (C - B^t)\left(\Sum_{i=0}^{l-2}C_{i} \Sum^{i}_{j=0}A^j\right)^t + (D-D^t).\\
\end{align*}

We divide the calculations in 4 steps:

 \begin{align*}
\intertext{1) Summands of the type $C_{l-2}(\cdots)C_{l-2}^t$. Denote by $\Sigma= \Sum_{j=0}^{l-1} A^j$. }
& C_{l-2} \left( \Sum_{j=0}^{l-2} A^j \right)(A - A^{-1}) \left(C_{l-2} \Sum_{j=0}^{l-2} A^j \right)^t + C_{l-2} \left(\Sum_{j=0}^{l-2} A^j \right)(- C_{l-2}^t) \\
& + C_{l-2}\left(C_{l-2} \Sum_{j=0}^{l-2} A^j \right)^t \\
&= C_{l-2} \left( (\Sigma - A^{-1} )(A - A^{-1})(\Sigma - A) + (\Sigma - A^{-1}) - (\Sigma - A)   \right)C_{l-2}^t \\
&= C_{l-2} \left( A - A^{-1} -A^{-1} +A \right)C_{l-2}^t = 0.\\
\intertext{2) Summands including $C_{l-2}$ and $B_{l-2}$. }
& C_{l-2} \left( \Sum_{j=0}^{l-2} A^j \right) B_{l-2} - \left( C_{l-2} \left( \Sum_{j=0}^{l-2} A^j \right)B_{l-2} \right)^t. \\
\intertext{3) Summands including $C_{l-2}$ or $B_{l-2}$ and terms of lower indices.}
&C_{l-2}\left( \Sum_{j=0}^{l-2} A^j \right)(A -
A^{-1})\left(\Sum_{i=0}^{l-3}C_i \Sum_{j=0}^{l} A^j \right)^t \\
&+ \left(  \Sum_{i=0}^{i}A^j \right)(A-A^{-1})\left( \Sum_{j=0}^{l-2} A^{-j} \right)C_{l-2}^t  \\
&+ C_{l-2}\left( \Sum_{j=0}^{l-2} A^j \right) \left( \Sum_{i=0}^{l-3}B_{i} - \Sum_{i=0}^{l-3}C_{i}^t  \right) + \left( \Sum_{i=0}^{l-3}\Sum_{j=0}^{i} C_{i}A^j \right)(B_{l-2} - C_{l-2}^t) \\
&+(C_{l-2} - B_{l-2}^t)\left( \Sum_{i=0}^{l-3}\Sum_{j=0}^{i} A^{-j} C_{i}^t \right) + \left( \Sum_{i=0}^{l-3} C_{i} - \Sum_{i=0}^{l-3}B_{i}^t \right)\left(\Sum_{j=0}^{l-2} A^{-j}\right)C_{l-2}^t \\
=&  \Sum_{i=0}^{l-3} C_{l-2} \left[   (\Sigma - A^{-1})(A - A^{-1})\left( \Sum_{j=0}^{i} A^{-j} \right) +  (\Sigma - A^{-1})(A^{-i-1}-1) + \left( \Sum_{j=0}^{i} A^{-j} \right) \right] C_{i}^t \\
&+  \Sum_{i=0}^{l-3} C_{i} \left[   \left( \Sum_{j=0}^{i} A^{j} \right)(A - A^{-1})(\Sigma - A) +\left( \Sum_{j=0}^{i} A^{-j} \right)(-1) + (1 - A^{i+1})(\Sigma - A)   \right] C_{l-2}^t \\
\end{align*}
\begin{align*}
& -\Sum_{i=0}^{l-3} B_{l-2}^t \left( \Sum_{j=0}^{i} A^{-j} \right) C_{i}^t +\Sum_{i=0}^{l-3} C_{i} \left( \Sum_{j=0}^{i} A^{j} \right) B_{l-2} \\
=& \Sum_{i=0}^{l-3} \left[ C_{l-2} \left( -A^{-1}(A + 1 - A^{-i} -A^{-i-1})-A^{-i-2} +A^{-1} + \Sum_{j=0}^{i}A^{-j}   \right)C_i^t  \right. \\
& \left. + C_i \left( (A^{i+1} +A^i - 1 - A^{-1})(-A) - \Sum_{j=0}^{i} A^j - A + A^{i+2}\right) C_{l-2}^t  \right] \\
& + \Sum_{i=0}^{l-3} \left[C_i \left(\Sum_{j=0}^{i}A^j\right)B_{l-2} - B_{l-2}^t\left( \Sum_{j=0}^{i}A^{-j} \right)C_i^t \right]  \\
=& \Sum_{i=0}^{l-3} \left[C_{l-2} \left(\Sum_{j=1}^{i+1}A^{-j}\right)C_{i}^t + C_{i}\left(- \Sum_{j=1}^{i+1}A^{j} \right)C_{l-2}^t \,\right] \\ 
& + \Sum_{i=0}^{l-3} \left[C_i \left(\Sum_{j=0}^{i}A^j\right)B_{l-2} - \left(C_i \left(\Sum_{j=0}^{i}A^j\right)B_{l-2}\right)^t\, \right] \\
=& \Sum_{i=0}^{l-3} \left[C_{l-2} \left(\Sum_{j=0}^{i}A^{j}\right)B_{i} + \left( C_{l-2} \left(\Sum_{j=0}^{i}A^{j}\right)B_{i}\right)^t \right.\\
&\left. + C_i \left(\Sum_{j=0}^{i}A^j\right)B_{l-2} - \left(C_i \left(\Sum_{j=0}^{i}A^j\right)B_{l-2}\right)^t \, \right]
\intertext{4) Summands without terms of index $l-2$.}
&\Sum_{i_1=0}^{l-3}\Sum_{i_2=0}^{l-3} C_{i_1} \left[  \left(\Sum_{j_1=0}^{i_i}A^{j_1}\right)(A-A^{-1})\left(\Sum_{j_2=0}^{i_2}A^{-j_2}\right) \right.\\ 
&\left. + \left(\Sum_{j_1=0}^{i_1}A^{j_1}\right)(A^{-i_2-1} - 1) + (1-A^{i_1+1})\left(\Sum_{j_2=0}^{i_2}A^{-j_2}\right)\right]C_{i_2}^t \\
=& \Sum_{i_1=0}^{l-3}\Sum_{i_2=0}^{l-3} C_{i_1} \left[ \Sum_{j_1=0}^{i_1-1}A^{j_1+1} + \Sum_{j_2=0}^{i_2}A^{i_1+1-j_2} - \Sum_{j_1=0}^{i_1} A^{j_1-1-i_2} \right.\\
&\left. - \Sum_{j_2=0}^{i_2-1}A^{-1-j_2} + \Sum_{j_1=0}^{i_1}A^{j_1-1-i_2} - \Sum_{j_1=0}^{i_1}A^{j_1} + \Sum_{j_2=0}^{i_2}A^{-j_2} - \Sum_{j_2=0}^{i_2} A^{i_1+1-j_2}\right]C_{i_2}^t\\
=& \Sum_{i_1=0}^{l-3}\Sum_{i_2=0}^{l-3} C_{i_1} \left[-1+1 \right]C_{i_2}^t =0.
\end{align*}

Thus, to sum up, we have
\begin{align*}
M'_{f\to f} =& D - D^t + C_{l-2} \left( \Sum_{j=0}^{l-2} A^j \right) B_{l-2} - \left( C_{l-2} \left( \Sum_{j=0}^{l-2} A^j \right)B_{l-2} \right)^t \\
& + \Sum_{i=0}^{l-3} \left[C_{l-2} \left(\Sum_{j=0}^{i}A^{j}\right)B_{i} + \left( C_{l-2} \left(\Sum_{j=0}^{i}A^{j}\right)B_{i}\right)^t \right.\\
&\left. + C_i \left(\Sum_{j=0}^{i}A^j\right)B_{l-2} - \left(C_i \left(\Sum_{j=0}^{i}A^j\right)B_{l-2}\right)^t \, \right].
\end{align*}

Hence, we keep all arrows in $Q_{f,f}$, and add an arrow for each composition $\gamma \alpha^i \beta$ where $\gamma\in Q_{f,m}$, $\alpha \in Q_{m,m}$, $\beta\in Q_{m,f}$ such that
\begin{itemize}
  \item neither $\gamma \alpha^i$ nor $\alpha^i \beta$ factors through an arrow in $Q_{f,f}$,
  \item either $\gamma \in C_{l-2}$ or $\beta \in B_{l-2}$, i.e.\ either $\gamma$ or $\beta$ has no extra relations with the minimal cycle of $Q_{m,m}$.
\end{itemize}
\item Finally, remove all loops and 2-cycles from the mutated quiver.
\end{itemize}


\begin{thebibliography}{BMRRT}

\bibitem[A]{A1}
Claire Amiot.
\newblock {On the structure of triangulated categories with finitely
  many indecomposables}.
\newblock {\em Bull. Soc. Math. France} 135(2007), no. 3, 435--474.

\bibitem[ARS]{ARS} M. Auslander, I. Reiten, S. Smal\o.
\newblock Representation theory of artin algebras.
\newblock Cambridge University Press 1995.

\bibitem[BG]{BG} K. Bongartz and P. Gabriel, 
\newblock Covering spaces in representation-theory
\newblock {\em Invent. Math.} 65 (1981/1982), no. 3, 331--378.

\bibitem[BIKR]{BIKR}
Igor Burban, Osamu Iyama, Bernhard Keller and Idun Reiten.
\newblock Cluster tilting for one-dimensional hypersurface singularities.
\newblock {\em Adv. Math.}, 217(6):2443--2484, 2008.

\bibitem[BIRSc]{BIRSc} 
Aslak~B. Buan, Osamu Iyama, Idun Reiten, and Jeanne Scott.
\newblock Cluster structures for 2-{C}alabi-{Y}au categories and unipotent groups.
\newblock {\em Compos. Math.} 145 (2009), no. 4, 1035--1079.

\bibitem[BIRSm]{BIRSm}
Aslak~B. Buan, Osamu Iyama, Idun Reiten, and David Smith.
\newblock Mutation of cluster-tilting objects and potentials.
\newblock Preprint, arXiv:0804.3813v3.

\bibitem[BMR1]{BMR1}
Aslak~Bakke Buan, Robert~J. Marsh, and Idun Reiten.
\newblock Cluster-tilted algebras of finite representation type.
\newblock {\em J. Algebra}, 306(2):412--431, 2006.

\bibitem[BMR2]{BMR2}
Aslak~Bakke Buan, Robert~J. Marsh, and Idun Reiten.
\newblock Cluster-tilted algebras.
\newblock {\em Trans. Amer. Math. Soc.}, 359(1):323--332 (electronic), 2007.

\bibitem[BMR3]{BMR3}
Aslak~Bakke Buan, Robert~J. Marsh, and Idun Reiten.
\newblock Cluster mutation via quiver representations.
\newblock {\em Comment. Math. Helv.}, 83(1):143--177, 2008.

\bibitem[BMRRT]{BMRRT}
Aslak~Bakke Buan, Robert Marsh, Markus Reineke, Idun Reiten, and Gordana
  Todorov.
\newblock Tilting theory and cluster combinatorics.
\newblock {\em Adv. Math.}, 204(2):572--618, 2006.

\bibitem[BOW1]{BOW1}
Marco~A. Bertani-{\O}kland, Steffen Opperman and Anette Wr{\aa}lsen.
\newblock Finding a cluster-tilting object for a representation finite cluster-tilted algebra.
\newblock Preprint, arXiv:0912.2911v1.

\bibitem[BOW2]{BOW2}
Marco~Angel Bertani-{\O}kland, Steffen Oppermann, and Anette Wr{\aa}lsen.
\newblock Graded mutation in cluster categories coming from hereditary categories with a tilting object.
\newblock In preparation.

\bibitem[CCS]{CCS1}
P. Caldero, F. Chapoton, and R. Schiffler.
\newblock Quivers with relations arising from clusters ({$A_n$} case).
\newblock {\em Trans. Amer. Math. Soc.} 358 (2006), no. 3, 1347--1364 (electronic).

\bibitem[DWZ]{DWZ} 
Harm Derksen, Jerzy Weyman, and Andrei Zelevinsky
\newblock Quivers with potentials and their representations. {I}. {M}utations
\newblock {\em Selecta Math. (N.S.)} 14 (2008), no. 1, 59--119.

\bibitem[FZ]{FZ1}
Sergey Fomin and Andrei Zelevinsky.
\newblock Cluster algebras. {I}. {F}oundations.
\newblock {\em J. Amer. Math. Soc.}, 15(2):497--529 (electronic), 2002.

\bibitem[G]{Gabriel} P. Gabriel.
\newblock The universal cover of a representation finite algebra
\newblock {\em Lecture Notes in Math.} 903 (1981), 68--105.

\bibitem[GLS]{GLS} {Christof Gei{\ss}, Bernard Leclerc, and  Jan Schr{\"o}er}.
\newblock Rigid modules over preprojective algebras
\newblock {\em Invent. Math.} 165 (2006), no. 3, 589--632.

\bibitem[HRS]{HRS} D. Happel, I. Reiten and S. O. Smal\o.
\newblock { Tilting in abelian categories and quasitilted algebras}, {\em Mem. Amer. Math. Soc.} 575 (1996).

\bibitem[H1]{Hubner} Thomas H\"ubner.
\newblock Exzeptionelle Vektorb\"undel und Reflektionen an Kippgarben \"uber Projektiven Gewichteten Kurven.
\newblock {\em Dissertation zur Erlangung des Doktorgrades des Fachbereichs Mathematik-Informatik der Universit\"at-Gesamthochschule Paderborn}, 1996.

\bibitem[H2]{Hubner2} Thomas H\"ubner.
\newblock Reflections and almost concealed canonical algebras.
\newblock {\em Sonderforschungsbereich 343, Diskrete Strukturen in der Mathematik} Universit\"at Bielefeld, 1997.

\bibitem[I]{I}
Osamu Iyama.
\newblock Higher dimensional Auslander-Reiten theory on maximal orthogonal subcategories.
\newblock {\em Adv. Math.} 210 (2007), no. 1, 22--50.

\bibitem[IY]{IY}
Osamu Iyama and Yuji Yoshino.
\newblock Mutation in triangulated categories and rigid Cohen-Macaulay
modules.
\newblock {\em Invent. Math.} 172 (2008), no. 1, 117--168.

\bibitem[K]{Keller} B. Keller.
\newblock On Triangulated orbit categories.
\newblock {\em Doc. Math.} 10 (2005), 551--581.

\bibitem[KR]{KR} B. Keller and I. Reiten.
\newblock Acyclic Calabi-Yau categories, with an appendix by Michel Van den Bergh.
\newblock {\em Compos. Math.} 144 (2008), no. 5, 1332--1348.

\bibitem[P]{Palu}
Yann Palu.
\newblock Grothendieck group and generalized mutation rule for
2-Calabi-Yau triangulated categories.
\newblock {\em J. Pure Appl. Algebra} 213 (2009) pp. 1438--1449.

\bibitem[R]{Ringel}
Claus~Michael Ringel.
\newblock The self-injective cluster-tilted algebras.
\newblock {\em Arch. Math. (Basel)} 91 (2008), no. 3, 218--225.

\bibitem[S]{S}
Ralf Schiffler.
\newblock A geometric model for cluster categories of type {$D_n$}.
\newblock {\em J. Algebraic Combin.} 27 (2008), no. 1, 1--21.

\bibitem[V]{Dagfinn}
Dagfinn Vatne.
\newblock The mutation class of {$D_n$} quivers.
\newblock {\em Comm. Algebra}, 38 (2010), no. 3, 1137--1146.

\bibitem[XZ]{XZ}
J. Xiao, Bin Zhu.
\newblock Locally finite triangulated categories.
\newblock {\em J. Algebra} 290 (2005), no. 2, 473--490.

\end{thebibliography}
\end{document}